\documentclass[12pt]{article}
\usepackage{amsmath}
\usepackage{amssymb}
\usepackage{tikz}

\usetikzlibrary{calc,fit,shapes.geometric}
\usepackage{subfigure}
\makeatletter
\def\fnum@figure#1{\figurename\nobreakspace\thefigure
       \hspace{0.6em}}                                     
\makeatother
\usepackage{marvosym}
\makeatletter

\newcommand{\Rmnum}[1]{\expandafter\@slowromancap\romannumeral #1@}
\makeatother

\begin{document}
\newtheorem{theorem}{Theorem}[section]
\newtheorem{corollary}{Corollary}
\newtheorem{definition}[theorem]{Definition}
\newtheorem{conjecture}[theorem]{Conjecture}
\newtheorem{problem}[theorem]{Problem}
\newtheorem{lemma}{Lemma}[section]
\newtheorem{proposition}[theorem]{Proposition}
\newtheorem{question}[theorem]{Question}
\newtheorem{claim}[theorem]{Claim}
\newtheorem{algorithm}[theorem]{Algorithm}
\newenvironment{proof}{\noindent {\bf Proof.}}
 {\hfill\rule{2mm}{2mm}\par}

\newcommand{\remark}{\medskip\par\noindent {\bf Remark.~~}}
\newcommand{\JCTB}{{\it J. Combin. Theory Ser. B.}, }
\newcommand{\JCT}{{\it J. Combin. Theory}, }
\newcommand{\JGT}{{\it J. Graph Theory}, }
\newcommand{\ComHung}{{\it Combinatorica}, }
\newcommand{\DAM}{{\it Discrete Applied Math.},}
\newcommand{\DM}{{\it Discrete Math.}, }
\newcommand{\ARS}{{\it Ars Combin.}, }
\newcommand{\SIAMDM}{{\it SIAM J. Discrete Math.}, }
\newcommand{\SIAMADM}{{\it SIAM J. Algebraic Discrete Methods}, }
\newcommand{\SIAMC}{{\it SIAM J. Comput.}, }
\newcommand{\ConAMS}{{\it Contemp. Math. AMS}, }
\newcommand{\TransAMS}{{\it Trans. Amer. Math. Soc.}, }
\newcommand{\AnDM}{{\it Ann. Discrete Math.}, }
\newcommand{\ConNum}{{\it Congr. Numer.}, }
\newcommand{\CJM}{{\it Canad. J. Math.}, }
\newcommand{\JLMS}{{\it J. London Math. Soc.}, }
\newcommand{\PLMS}{{\it Proc. London Math. Soc.}, }
\newcommand{\PAMS}{{\it Proc. Amer. Math. Soc.}, }
\newcommand{\JCMCC}{{\it J. Combin. Math. Combin. Comput.}, }

\long\def\longdelete#1{} \baselineskip 21 pt  

\title{\bf  Power domination in regular claw-free graphs
\thanks{Supported in part by National Natural Science Foundation of China (Nos.~11371008 and 91230201).}}
\author{Changhong Lu \ \ Rui Mao \ \  Bing Wang\\
 \normalsize Department of Mathematics,\\
\normalsize Shanghai Key Laboratory of PMMP,\\
\normalsize East China Normal University,\\
\normalsize Shanghai 200241, P. R. China\\
\normalsize C. Lu email: chlu@math.ecnu.edu.cn\\
\normalsize R. Mao(\Letter) email: maorui1111@163.com\\
\normalsize B. Wang email: wuyuwuyou@126.com\\}

\maketitle

\begin{abstract}

In this paper, we first
show that the power domination number of a connected $4$-regular claw-free graph on $n$ vertices is at most $\frac{n+1}{5}$, and the bound is sharp. The statement partly disprove the conjecture presented by Dorbec et al. in SIAM J. Discrete Math., 27:1559-1574, 2013. Then we present a dynamic programming style linear-time algorithm for weighted power domination problem in trees.
\bigskip

\end{abstract}
\section{Introduction}                        
Electric power systems need to be continually monitored. One way to fulfill this task is to place phase measurement units
at selected locations in the system. The power system monitoring problem, as introduced in \cite{Baldwin}, asks for as few as possible measurement
devices to be put in an electric power system. The power system monitoring problem was then described as a graph theoretical problem in
\cite{Haynes1}. The problem is similar to a problem of domination, in which, additionally, the possibility of some propagation according to
Kirschoff laws is considered.

Let $G=(V, E)$ be a connected, simple graph with vertex set $V=V(G)$ and edge set $E=E(G)$. The \emph{open neighborhood} of a vertex $v\in V(G)$ is the set $N_G(v)=\{u\in V(G)|uv\in E(G)\}$, and the \emph{degree} of $v$ is $d_G(v)=|N_G(v)|$. The
\emph{closed neighborhood} of $v$ is the set $N_G[v]=N_G(v)\cup\{v\}$. A graph $G$ is $k$-\emph{regular} if $d_G(v) = k$ for every vertex $v \in V(G)$. The \emph{open neighborhood of a
subset} $S\subseteq V$ of vertices is the set $N_G(S) =\cup_{v\in S}N(v)$, while the \emph{closed neighborhood}
of $S$ is the set $N_G[S]=N_G(S)\cup S$. We denote $K_{i,j}$ the complete bipartite graph with two partite sets of cardinality $i$ and $j$, respectively. A \emph{claw-free} graph is a graph that does not contain a claw, i.e., $K_{1,3}$, as an induced subgraph.
We say a subset of $V(G)$ an \emph{independent set} if no two vertices of the set are adjacent in $G$. Let $x$ and $y$ are two vertices of $G$. Denote by $d(x,y)$ the \emph{distance} of $x$ and $y$ in $G$. We say a subset of $V(G)$ a \emph{packing} if no two vertices in the set of distance less than three in $G$.
For two graphs $G=(V,E)$ and $G'=(V',E')$, let $G\cup G'=(V\cup V',E\cup E')$ and $G\cap G'=(V\cap V',E\cap E')$. If $G\cup G'=\emptyset$, then $G$ and $G'$ are disjoint.
 For any vertex subset $X$ of $G$, let $G-X=G[V\setminus X]$ and for $X=\{x\}$ let $G-x=G-\{x\}$ for short. For notation and graph theory terminology not defined herein, we in general follow \cite {Diestel}.

The original definition of power domination was simplified to the following definition
independently in \cite{Dorbec1, Dorfling, Guo, Liao} and elsewhere.

\begin{definition}\label{def1}
Let $G$ be a graph. A set $S\subseteq V(G)$ is a power dominating set (abbreviated
as PDS) of $G$ if and only if all vertices of $V(G)$ have messages either by Observation Rule 1 (abbreviated as OR 1) initially or by Observation Rule 2 (abbreviated as OR 2) recursively.

{\bf OR 1.} A vertex $v\in S$ sends a message to itself
and all its neighbors. We say that $v$ observes itself and all its neighbors.

{\bf OR 2.} If an observed vertex $v$ has only one
unobserved neighbor $u$, then $v$ will send a message to $u$. We say that $v$ observes $u$.
\end{definition}

 Let $G=(V, E)$ be a graph and $S$ be a subset of $V$. For $i\geq0$, we define the set $P_G^{i}(S)$ of vertices observed by $S$ at step $i$ by the following rules:

(1) $P_G^{0}(S)= N[S]$;

(2) $P_G^{i+1}(S)=\cup \{N_G[v]: v\in P_G^{i}(S)$ such that $|N_G[v] \setminus P_G^{i}(S)|\leq 1\}$.

Note that if $S$ is a power domination set of $G$, then  there is a minimal integer $i_0$ such that $P_G^{i_0}(S) =V(G)$. Hence $P_G^{j} (S)= P_G^{i_0} (S)$ for every $j \geq i_0$ and we accordingly define $P_G^{\infty} (S)= P_G^{i_0} (S)$. If the graph $G$ is clear from the context, we will remove the subscripts $G$ for short.

The \emph{power domination number} of a graph $G$, denoted by $\gamma_p(G)$, is the minimum cardinality of a PDS of $G$. A PDS of $G$ with minimum cardinality is called a \emph{$\gamma_p(G)$-set}. The power domination problem was known to be NP-complete even for bipartite graphs, planar graphs and split graphs, see \cite{Guo, Haynes1}.

Chang et al \cite{Chang} generalized the power domination to $k$-power domination by replacing the OR 2 with the following observation rule: If an observed vertex $v$ has at most $k$ unobserved neighbors, then $v$ will
send a message to all its unobserved neighbors. The definition also can be found in \cite{Wang}.

 The \emph{$k$-power domination number} of $G$, denoted by $\gamma_{p,k}(G)$, is the minimum cardinality
of a $k$-power dominating set of $G$. When $k=1$,
The $k$-power domination is usual power domination. Both power domination and $k$-power domination are now well-studied in the literature (see, for example, \cite{Aazami, Chang, Dorbec, Dorbec1, Dorfling, Guo, Wang, Xu, Zhao}).

The paper is organized as follows.
In section 2, we will give sharp upper bounds for the power domination number of a connected 4-regular claw-free graph on $n$ vertices. The result will partly disprove the conjecture introduced by Dorbec et al. in \cite{Dorbec}. In section 3, we give a dynamic programming style linear-time algorithm for weighted power domination in trees.

\section{Power domination in 4-regular claw-free graphs}
Zhao et al. \cite{Zhao} proved that if $G$ is a connected claw-free cubic graph of order $n$, then $\gamma_{p}(G)\le\frac{n}{4}$.
Chang et al. \cite{Chang} showed that if $G$ is a connected claw-free $(k+2)$-regular graph on $n$ vertices, then $\gamma_{p,k}(G)\le \frac {n}{k+3}$. Recently, Dorbec et al. \cite{Dorbec} gave an found that the claw-free condition can be removed. Then they presented the following conjecture.

\begin{conjecture}\label{conj1}\emph{(\cite{Dorbec})}
For $k\ge1$ and $r\ge 3$, if $G \ncong K_{r,r}$ is a connected $r$-regular graph
of order $n$, then $\gamma_{p,k}(G)\le \frac{n}{r + 1}$.
\end{conjecture}

It is obvious that if the conjecture holds for $k=1$, then it also holds for all $k\ge 2$. Hence, we pay our attention to the case of $k=1$ in the context. Dorbec et al. \cite{Dorbec} showed that $\gamma_{p,k}(G)\le \frac {n}{k+3}$ for any connected $(k+2)$-regular graph $G$ on $n$ vertices. It means that Conjecture~\ref{conj1} holds for $k=1$ and $r=3$. However, for each even $r\ge 4$, we show that Conjecture \ref{conj1} does not always hold.

We first give a counterexample $E_0$ with $2r+1$ vertices. Then based on $E_0$, we obtain infinitely many counterexamples of Conjecture \ref{conj1} (in fact, the graphs are extremal graphs of Theorem~\ref{mainth} as well). Pick first two copies of $K_r$ and one singleton vertex $u$, then add $r$ independent edges between the two copies and $r$ edges linking the vertex $u$ to these vertices which are not incident to any independent edge before. Denote by the resulting $r$-regular graph $E_0$. Let $E_1$ be the graph obtained from $E_0$ by splitting the vertex $u$ into two vertices of degree $\frac{r}{2}$, say $u_1^1$ and $u_2^1$, then add an extra $K_r$ and link exactly $\frac{r}{2}$ edges from each $u_i^1$ to the new $K_r$ such that the resulting graph is $r$-regular as well. Similarly, let $E_j$ be the graph obtained from $E_{j-1}$ by splitting the vertex linked to two copies of $K_r$ into two vertices of degree $\frac{r}{2}$, say $u_1^j$ and $u_2^j$, then add an extra $K_r$ and link exactly $\frac{r}{2}$ edges from each $u_i^j$ to the new $K_r$ such that the resulting graph is $r$-regular as well, here $i=1,2$. The case for $r=4$ is shown in Figure 1.

It is obvious that $\gamma_p(E_0)=2$. Note that for any integer $k\ge1$, $\gamma_p(E_k)\ge k+2$ since any power domination set of $E_k$ must contain at least one vertex in each added $K_r$ and also contain at least two vertices in the rest part isomorphic to the graph $E_1$ minus the added $K_r$. On the other hand, two vertices $a, u_2^k$ together with $k$ vertices, picking exactly one vertex from each added $K_r$ of $E_k$, form a power domination set of $E_k$. Since $|V(E_k)|=2r+1+k(r+1)$, $\gamma_p(E_k)=k+2=\frac{2r+2+k(r+1)}{r+1}$.

The main result of the paper is as follows.
\begin{theorem}\label{mainth}
Let $G$ be a connected claw-free $4$-regular graph of order $n$. Then $\gamma_{P}(G)\le\frac{n+1}{5}$ and the bound is sharp.
\end{theorem}

\begin{center}
\setlength{\unitlength}{0.20mm}
\begin{picture}(647,214)
\put(112,97){\circle*{6}}
\put(60,193){\circle*{6}}
\put(166,193){\circle*{6}}
\put(60,133){\circle*{6}}
\qbezier(60,193)(60,163)(60,133)
\qbezier(60,193)(113,193)(166,193)
\put(166,133){\circle*{6}}
\qbezier(166,193)(166,163)(166,133)
\put(165,133){\circle*{6}}
\qbezier(60,133)(112,133)(165,133)
\put(102,132){\circle*{6}}
\qbezier(60,193)(81,163)(102,132)
\put(102,194){\circle*{6}}
\qbezier(60,133)(81,164)(102,194)
\qbezier(102,194)(102,163)(102,132)
\put(128,194){\circle*{6}}
\put(128,132){\circle*{6}}
\qbezier(128,194)(128,163)(128,132)
\qbezier(128,194)(146,164)(165,133)
\qbezier(166,193)(147,163)(128,132)
\put(248,194){\circle*{6}}
\put(354,194){\circle*{6}}
\put(248,134){\circle*{6}}
\put(354,134){\circle*{6}}
\put(353,134){\circle*{6}}
\put(290,133){\circle*{6}}
\put(290,195){\circle*{6}}
\put(316,195){\circle*{6}}
\put(316,134){\circle*{6}}
\qbezier(248,194)(248,164)(248,134)
\qbezier(248,194)(301,194)(354,194)
\qbezier(354,194)(354,164)(354,134)
\qbezier(248,134)(300,134)(353,134)
\qbezier(248,194)(269,164)(290,133)
\qbezier(248,134)(269,165)(290,195)
\qbezier(290,195)(290,164)(290,133)
\qbezier(316,195)(316,165)(316,134)
\qbezier(316,195)(334,165)(353,134)
\qbezier(354,194)(335,164)(316,134)
\put(258,92){\circle*{6}}
\put(343,96){\circle*{6}}
\qbezier(166,133)(139,115)(112,97)
\qbezier(60,133)(86,115)(112,97)
\qbezier(166,193)(213,90)(112,97)
\qbezier(60,193)(10,88)(112,97)
\qbezier(248,134)(253,113)(258,92)
\qbezier(353,134)(348,115)(343,96)
\qbezier(248,194)(209,126)(258,92)
\qbezier(354,194)(397,133)(343,96)
\put(280,114){\circle*{6}}
\put(280,74){\circle*{6}}
\qbezier(280,114)(280,94)(280,74)
\put(321,114){\circle*{6}}
\qbezier(280,114)(300,114)(321,114)
\put(321,74){\circle*{6}}
\qbezier(321,114)(321,94)(321,74)
\qbezier(280,74)(300,74)(321,74)
\qbezier(258,92)(269,103)(280,114)
\qbezier(258,92)(269,83)(280,74)
\qbezier(321,114)(332,105)(343,96)
\qbezier(343,96)(332,85)(321,74)
\qbezier(280,114)(300,94)(321,74)
\qbezier(321,114)(300,94)(280,74)
\put(109,81){$u$}
\put(99,203){$a$}
\put(220,91){$u_1^1$}
\put(350,94){$u_2^1$}
\put(286,204){$a$}
\put(96,31){$E_0$}
\put(293,31){$E_1$}
\put(438,196){\circle*{6}}
\put(563,196){\circle*{6}}
\qbezier(438,196)(500,196)(563,196)
\put(438,141){\circle*{6}}
\qbezier(438,196)(438,169)(438,141)
\put(563,141){\circle*{6}}
\qbezier(438,141)(500,141)(563,141)
\qbezier(563,196)(563,169)(563,141)
\put(474,196){\circle*{6}}
\put(474,140){\circle*{6}}
\qbezier(474,196)(474,168)(474,140)
\put(525,195){\circle*{6}}
\put(525,141){\circle*{6}}
\qbezier(525,195)(525,168)(525,141)
\qbezier(438,196)(456,168)(474,140)
\qbezier(474,196)(456,169)(438,141)
\qbezier(525,195)(544,168)(563,141)
\qbezier(563,196)(544,169)(525,141)
\put(418,94){\circle*{6}}
\qbezier(438,141)(428,118)(418,94)
\qbezier(438,196)(401,122)(418,94)
\put(591,94){\circle*{6}}
\qbezier(563,141)(577,118)(591,94)
\qbezier(563,196)(607,153)(591,94)
\put(437,113){\circle*{6}}
\qbezier(418,94)(427,104)(437,113)
\put(437,75){\circle*{6}}
\qbezier(418,94)(427,85)(437,75)
\qbezier(437,113)(437,94)(437,75)
\put(460,113){\circle*{6}}
\qbezier(437,113)(448,113)(460,113)
\put(460,75){\circle*{6}}
\qbezier(460,113)(460,94)(460,75)
\qbezier(437,75)(448,75)(460,75)
\qbezier(437,113)(448,94)(460,75)
\qbezier(460,113)(448,94)(437,75)
\put(473,95){\circle*{6}}
\qbezier(460,113)(466,104)(473,95)
\qbezier(473,95)(466,85)(460,75)
\put(485,113){\circle*{6}}
\qbezier(473,95)(479,104)(485,113)
\put(485,75){\circle*{6}}
\qbezier(473,95)(479,85)(485,75)
\qbezier(485,113)(485,94)(485,75)
\put(508,113){\circle*{6}}
\qbezier(485,113)(496,113)(508,113)
\put(508,75){\circle*{6}}
\qbezier(508,113)(508,94)(508,75)
\qbezier(485,75)(496,75)(508,75)
\qbezier(485,113)(496,94)(508,75)
\qbezier(508,113)(496,94)(485,75)
\put(521,95){\circle*{6}}
\qbezier(508,113)(514,104)(521,95)
\qbezier(521,95)(514,85)(508,75)
\put(542,95){\circle*{6}}
\put(553,112){\circle*{6}}
\qbezier(542,95)(547,104)(553,112)
\put(553,76){\circle*{6}}
\qbezier(542,95)(547,86)(553,76)
\qbezier(553,112)(553,94)(553,76)
\put(573,112){\circle*{6}}
\qbezier(553,112)(563,112)(573,112)
\put(573,76){\circle*{6}}
\qbezier(573,112)(573,94)(573,76)
\qbezier(553,76)(563,76)(573,76)
\qbezier(553,112)(563,94)(573,76)
\qbezier(573,112)(563,94)(553,76)
\qbezier(573,112)(582,103)(591,94)
\qbezier(591,94)(582,85)(573,76)
\put(523,94){$...$}
\put(392,94){$u_1^k$}
\put(599,93){$u_2^k$}
\put(471,205){$a$}
\put(494,32){$E_k$}
\put(201,-9){Figure 1. Counterexample graphs.}
\end{picture}
\end{center}

\subsection{Structure of the minimal counterexample $G$}
If the statement of Theorem \ref{mainth} fails, then we suppose that $G$ is a counterexample with minimal $|V(G)|$. We have the following statement.

\begin{lemma}\label{counterex}
$G$ is neither isomorphic to $K_5$ nor $I_i$ for $i\in\{1,2,\cdots,8\}$.
\end{lemma}

\begin{center}
\setlength{\unitlength}{0.20mm}
\begin{picture}(571,349)
\put(146,150){\circle*{6}}
\put(146,108){\circle*{6}}
\put(276,150){\circle*{6}}
\put(276,108){\circle*{6}}
\put(189,151){\circle*{6}}
\put(189,107){\circle*{6}}
\put(234,150){\circle*{6}}
\put(234,107){\circle*{6}}
\qbezier(146,150)(146,129)(146,108)
\qbezier(146,150)(211,150)(276,150)
\qbezier(146,108)(211,108)(276,108)
\qbezier(276,150)(276,129)(276,108)
\qbezier(189,151)(189,129)(189,107)
\qbezier(146,150)(167,129)(189,107)
\qbezier(146,108)(167,130)(189,151)
\qbezier(234,150)(234,129)(234,107)
\qbezier(234,150)(255,129)(276,108)
\qbezier(276,150)(255,129)(234,107)
\qbezier(146,150)(216,182)(276,150)
\qbezier(146,108)(214,72)(276,108)
\put(420,303){\circle*{6}}
\put(420,249){\circle*{6}}
\qbezier(420,303)(420,276)(420,249)
\put(393,303){\circle*{6}}
\put(364,274){\circle*{6}}
\qbezier(393,303)(378,289)(364,274)
\put(394,248){\circle*{6}}
\qbezier(364,274)(379,261)(394,248)
\put(421,275){\circle*{6}}
\qbezier(393,303)(407,289)(421,275)
\qbezier(421,275)(407,262)(394,248)
\put(140,205){\circle{0}}
\put(364,303){\circle*{6}}
\put(392,331){\circle*{6}}
\put(420,303){\circle*{6}}
\qbezier(392,331)(406,317)(420,303)
\put(364,249){\circle*{6}}
\qbezier(364,249)(392,249)(420,249)
\put(420,250){\circle*{6}}
\qbezier(364,249)(395,221)(420,250)
\put(472,311){\circle*{6}}
\put(472,255){\circle*{6}}
\qbezier(472,311)(472,283)(472,255)
\put(530,311){\circle*{6}}
\qbezier(472,311)(501,311)(530,311)
\put(530,255){\circle*{6}}
\qbezier(530,311)(530,283)(530,255)
\qbezier(472,255)(501,255)(530,255)
\qbezier(472,311)(443,281)(472,255)
\qbezier(472,311)(504,340)(530,311)
\qbezier(530,311)(569,282)(530,255)
\qbezier(472,255)(506,219)(530,255)
\put(501,312){\circle*{6}}
\put(472,283){\circle*{6}}
\qbezier(501,312)(486,298)(472,283)
\put(502,255){\circle*{6}}
\qbezier(472,283)(487,269)(502,255)
\put(530,283){\circle*{6}}
\qbezier(501,312)(515,298)(530,283)
\qbezier(530,283)(516,269)(502,255)
\put(285,320){\circle*{6}}
\put(247,257){\circle*{6}}
\qbezier(285,320)(266,289)(247,257)
\put(327,257){\circle*{6}}
\qbezier(285,320)(306,289)(327,257)
\qbezier(247,257)(287,257)(327,257)
\put(267,291){\circle*{6}}
\put(304,291){\circle*{6}}
\qbezier(267,291)(285,291)(304,291)
\put(289,257){\circle*{6}}
\qbezier(267,291)(278,274)(289,257)
\qbezier(304,291)(296,274)(289,257)
\qbezier(285,320)(335,308)(327,257)
\qbezier(285,320)(237,304)(247,257)
\qbezier(247,257)(292,221)(327,257)
\qbezier(364,303)(364,276)(364,249)
\qbezier(364,303)(392,303)(420,303)
\qbezier(364,303)(393,320)(420,303)
\qbezier(392,331)(324,307)(364,249)
\qbezier(392,331)(462,311)(420,250)
\put(20,307){\circle*{6}}
\put(20,257){\circle*{6}}
\qbezier(20,307)(20,282)(20,257)
\put(141,307){\circle*{6}}
\qbezier(20,307)(80,307)(141,307)
\put(141,257){\circle*{6}}
\qbezier(20,257)(80,257)(141,257)
\qbezier(141,307)(141,282)(141,257)
\put(61,306){\circle*{6}}
\put(61,257){\circle*{6}}
\qbezier(61,306)(61,282)(61,257)
\put(99,306){\circle*{6}}
\put(99,257){\circle*{6}}
\qbezier(99,306)(99,282)(99,257)
\put(225,307){\circle*{6}}
\qbezier(141,307)(183,307)(225,307)
\put(225,257){\circle*{6}}
\qbezier(141,257)(183,257)(225,257)
\qbezier(225,307)(225,282)(225,257)
\qbezier(20,307)(40,282)(61,257)
\qbezier(61,306)(40,282)(20,257)
\qbezier(99,306)(120,282)(141,257)
\qbezier(141,307)(120,282)(99,257)
\put(183,306){\circle*{6}}
\put(183,256){\circle*{6}}
\qbezier(183,306)(183,281)(183,256)
\qbezier(183,306)(204,282)(225,257)
\qbezier(225,307)(204,282)(183,256)
\qbezier(20,307)(128,338)(225,307)
\qbezier(20,257)(130,227)(225,257)
\qbezier(392,331)(378,317)(364,303)
\put(305,152){\circle*{6}}
\put(346,152){\circle*{6}}
\qbezier(305,152)(325,152)(346,152)
\put(305,109){\circle*{6}}
\qbezier(305,152)(305,131)(305,109)
\put(346,109){\circle*{6}}
\qbezier(305,109)(325,109)(346,109)
\qbezier(346,152)(346,131)(346,109)
\put(362,131){\circle*{6}}
\qbezier(346,152)(354,142)(362,131)
\qbezier(362,131)(354,120)(346,109)
\put(377,152){\circle*{6}}
\qbezier(362,131)(369,142)(377,152)
\put(377,109){\circle*{6}}
\qbezier(362,131)(369,120)(377,109)
\put(417,152){\circle*{6}}
\qbezier(377,152)(397,152)(417,152)
\qbezier(377,152)(377,131)(377,109)
\put(417,109){\circle*{6}}
\qbezier(417,152)(417,131)(417,109)
\qbezier(377,109)(397,109)(417,109)
\qbezier(377,152)(397,131)(417,109)
\qbezier(417,152)(397,131)(377,109)
\qbezier(305,152)(363,181)(417,152)
\qbezier(305,109)(367,76)(417,109)
\qbezier(305,152)(325,131)(346,109)
\qbezier(346,152)(325,131)(305,109)
\put(453,126){\circle*{6}}
\put(453,93){\circle*{6}}
\qbezier(453,126)(453,110)(453,93)
\put(485,126){\circle*{6}}
\qbezier(453,126)(469,126)(485,126)
\put(485,93){\circle*{6}}
\qbezier(485,126)(485,110)(485,93)
\qbezier(453,93)(469,93)(485,93)
\qbezier(453,126)(469,110)(485,93)
\qbezier(485,126)(469,110)(453,93)
\put(513,126){\circle*{6}}
\qbezier(485,126)(499,126)(513,126)
\put(546,93){\circle*{6}}
\qbezier(485,93)(515,93)(546,93)
\put(513,93){\circle*{6}}
\qbezier(513,126)(513,110)(513,93)
\put(546,126){\circle*{6}}
\qbezier(513,126)(529,126)(546,126)
\qbezier(546,126)(546,110)(546,93)
\qbezier(513,126)(529,110)(546,93)
\qbezier(546,126)(529,110)(513,93)
\put(471,166){\circle*{6}}
\put(471,141){\circle*{6}}
\qbezier(471,166)(471,154)(471,141)
\put(520,166){\circle*{6}}
\qbezier(471,166)(495,166)(520,166)
\put(520,141){\circle*{6}}
\qbezier(471,141)(495,141)(520,141)
\qbezier(520,166)(520,154)(520,141)
\qbezier(471,166)(495,154)(520,141)
\qbezier(471,141)(495,154)(520,166)
\qbezier(471,141)(462,134)(453,126)
\qbezier(520,141)(533,134)(546,126)
\qbezier(471,166)(427,127)(453,93)
\qbezier(520,166)(571,141)(546,93)
\put(113,213){$I_1$}
\put(278,213){$I_2$}
\put(384,213){$I_3$}
\put(494,213){$I_4$}
\put(57,62){$I_5$}
\put(207,62){$I_6$}
\put(355,62){$I_7$}
\put(21,150){\circle*{6}}
\put(115,150){\circle*{6}}
\put(64,149){\circle*{6}}
\put(64,105){\circle*{6}}
\qbezier(64,149)(64,127)(64,105)
\put(21,106){\circle*{6}}
\put(115,106){\circle*{6}}
\qbezier(21,106)(68,106)(115,106)
\qbezier(21,150)(21,128)(21,106)
\qbezier(115,150)(115,128)(115,106)
\qbezier(21,150)(68,182)(115,150)
\qbezier(21,106)(68,69)(115,106)
\put(1,125){\circle*{6}}
\qbezier(21,150)(11,138)(1,125)
\qbezier(1,125)(11,116)(21,106)
\qbezier(51,125)(57,115)(64,105)
\put(94,125){\circle*{6}}
\put(50,125){\circle*{6}}
\qbezier(51,125)(57,137)(64,149)
\qbezier(1,125)(26,125)(51,125)
\qbezier(50,125)(72,125)(94,125)
\qbezier(1,125)(53,156)(94,125)
\put(488,62){$I_8$}
\qbezier(94,125)(104,116)(115,106)
\qbezier(21,150)(68,150)(115,150)
\qbezier(94,125)(104,138)(115,150)
\put(126,24){Figure 2. $I_i$ for $i\in\{1,2,\cdots,8\}$.}
\end{picture}
\end{center}

Let $H$ be a subgraph of $G$. We say $v\in V(H)$ a \emph{saturated vertex} of $H$ if $d_{H}(v)=d_{G}(v)=4$. If there is an induced cycle such that all of the vertices are saturated vertices of $H$, then we say the cycle a \emph{saturated cycle} of $H$. Especially, a saturated triangle (or quadrilateral) is a saturated cycles of order three (or four). For convenience, if $G$ contains a subgraph isomorphic to $H$, then we only say that $G$ contains $H$.


For any integer $k\ge 2$, let $L_k$ be the graph obtained from $k$ disjoint copies of $K_4$ in linear order, say $D_1,D_2,\cdots,D_k$, by linking any two adjacent copies  $(D_i, D_{i+1})$ with two independent edges, where $i=1,\cdots, k-1$ (see Figure 3).

\begin{center}
\setlength{\unitlength}{0.2mm}
\begin{picture}(566,110)
\put(0,99){\circle*{6}}
\put(0,20){\circle*{6}}
\put(77,99){\circle*{6}}
\qbezier(0,99)(38,99)(77,99)
\put(77,20){\circle*{6}}
\qbezier(0,20)(38,20)(77,20)
\qbezier(77,99)(77,60)(77,20)
\put(156,99){\circle*{6}}
\qbezier(77,99)(116,99)(156,99)
\qbezier(0,99)(38,60)(77,20)
\qbezier(77,99)(38,60)(0,20)
\put(156,99){\circle*{6}}
\put(156,20){\circle*{6}}
\put(236,99){\circle*{6}}
\put(236,20){\circle*{6}}
\put(265,20){\circle*{0}}
\qbezier(156,99)(196,99)(236,99)
\qbezier(156,20)(196,20)(236,20)
\qbezier(236,99)(236,60)(236,20)
\qbezier(265,20)(250,20)(236,20)
\qbezier(156,99)(196,60)(236,20)
\qbezier(236,99)(196,60)(156,20)
\qbezier(0,99)(0,60)(0,20)
\qbezier(156,99)(156,60)(156,20)
\put(274,58){\circle*{2}}
\put(284,58){\circle*{2}}
\put(293,58){\circle*{2}}
\qbezier(77,20)(116,20)(156,20)
\qbezier(236,99)(251,99)(267,99)
\put(330,99){\circle*{6}}
\put(330,20){\circle*{6}}
\put(407,99){\circle*{6}}
\put(407,20){\circle*{6}}
\put(486,99){\circle*{6}}
\put(486,20){\circle*{6}}
\put(566,99){\circle*{6}}
\put(566,20){\circle*{6}}
\qbezier(330,99)(368,99)(407,99)
\qbezier(330,20)(368,20)(407,20)
\qbezier(407,99)(407,60)(407,20)
\qbezier(407,99)(446,99)(486,99)
\qbezier(330,99)(368,60)(407,20)
\qbezier(407,99)(368,60)(330,20)
\qbezier(486,99)(526,99)(566,99)
\qbezier(486,20)(526,20)(566,20)
\qbezier(566,99)(566,60)(566,20)
\qbezier(486,99)(526,60)(566,20)
\qbezier(566,99)(526,60)(486,20)
\qbezier(330,99)(330,60)(330,20)
\qbezier(486,99)(486,60)(486,20)
\qbezier(407,20)(446,20)(486,20)
\qbezier(330,99)(314,99)(298,99)
\qbezier(330,20)(314,20)(298,20)
\put(22,0){$D_1$}\put(182,0){$D_2$}\put(362,0){$D_{k-1}$}\put(512,0){$D_{k}$}
\put(160,-30){Figure 3. $L_k$ for some $k\ge 2$.}
\end{picture}
\end{center}
\vskip 1em

\begin{lemma}\label{le1}
$G$ does not contain $L_k$ for $k\ge3$.
\end{lemma}
\begin{proof}
Otherwise, suppose that $G$ contains a subgraph $H$ isomorphic to $L_3$. Let $v_1, v_2\in V(D_1)$, $v_3, v_4\in V(D_3)$ be the four vertices of degree three in $H$.

Suppose first that $H$ is an induced subgraph of $G$.  Let $G'$ be obtained from $G$ by deleting all saturated vertices in subgraph $H$ of $G$ and adding four extra edges $v_1v_3, v_1v_4, v_2v_3, v_2v_4$. Then $G'$ is also a connected claw-free $4$-regular graph of order $n'$. Since $G'$ is not a counterexample, $\gamma_p(G')\leq \frac {n'+1}5$. Let $M'$ be a $\gamma_p(G')$-set of $G'$. If one of vertices $v_1,v_2,v_3$ and $v_4$ is in $M'$, say $v_1\in M'$, then $M=M'\cup\{x\}$ is a PDS of $G$, where $x\in V(D_3)\setminus \{v_3,v_4\}$. Suppose now that $\{v_1, v_2, v_3, v_4\}\cap M'=\emptyset$. Then $M'$ together with one vertex of $D_2$ forms a PDS of $G$. Note that $|V(G')|=n-6$, then $\gamma_p(G)\leq \mid M\mid \leq \frac {n'+1}5 +1<\frac {n+1}5$, a contradiction.

Now suppose that $H$ is not an induced subgraph of $G$. If there are two extra edges linked four vertices $v_1,v_2,v_3$ and $v_4$, say $v_1v_3, v_2v_4\in E(G)$ then $G\cong I_1$. This contradicts Lemma \ref{counterex}.
 Then without loss of generality suppose that $v_1v_3\in E(G)$ but $v_2v_4\notin E(G)$. We can obtain $G'$ from $G$ by exactly replacing the subgraph $H$ with a subgraph $H'$ isomorphic to $K_5-e$ (the complete graph $K_5$ minus an edge) by identical their two vertices of degree three. Let $M'$ be a $\gamma_p(G')$-set of $G'$. It is easy to check that $V(H')\cap M'\neq \emptyset$. Then $M=M'\setminus V(H')\cup\{v_2, v_4\}$ is a PDS of $G$. Since $|V(G')|=n-7$, $\gamma_p(G)\leq |M| \leq \frac {n'+1}5 +1<\frac {n+1}5$, a contradiction.

Note that if $G$ contains a subgraph $L_k$ for some $k\ge4$, then $G$ contains an induced subgraph isomorphic to $L_3$. Therefore $G$ does not contain any subgraph isomorphic to $L_k$ for $k\ge 3$.
\end{proof}

\noindent
{\bf Remark.} From the above proof, we note that $G$ has some forbidden subgraphs, such as $L_k$ for $k\ge 3$. In fact, there are other forbidden subgraphs in $G$. That is, if $G$ contains such a subgraph, then we can replace the subgraph by some smaller subgraph to obtain a new graph $G'$. There is a power domination set $S'$ of $G'$ such that $|S'|\le \frac{|V(G')|+1}{5}$. Based on $S'$, we can get a desired power domination set of $G$ such that $\gamma_p(G)\le \frac{|V(G)|+1}{5}$, contradicting that $G$ is a counterexample.

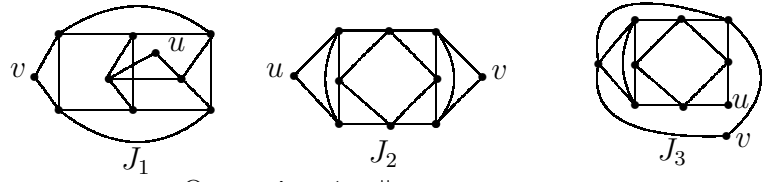
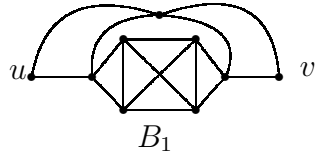
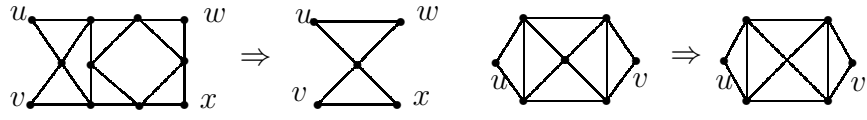
\begin{center}
\setlength{\unitlength}{0.23mm}
\begin{picture}(479,243)
\put(247,210){\circle*{4}}
\put(247,157){\circle*{4}}
\qbezier(247,210)(247,184)(247,157)
\put(220,211){\circle*{4}}
\put(191,182){\circle*{4}}
\qbezier(220,211)(205,197)(191,182)
\put(221,156){\circle*{4}}
\qbezier(191,182)(206,169)(221,156)
\put(248,183){\circle*{4}}
\qbezier(220,211)(234,197)(248,183)
\qbezier(248,183)(234,170)(221,156)
\put(191,211){\circle*{4}}
\put(274,184){\circle*{4}}
\qbezier(274,184)(260,198)(247,211)
\put(191,157){\circle*{4}}
\qbezier(191,157)(219,157)(247,157)
\put(247,157){\circle*{4}}
\qbezier(191,211)(191,184)(191,157)
\qbezier(191,211)(219,211)(247,211)
\put(65,131){$J_1$}
\put(29,209){\circle*{4}}
\put(117,209){\circle*{4}}
\put(72,208){\circle*{4}}
\put(72,165){\circle*{4}}
\qbezier(72,208)(72,187)(72,165)
\put(29,165){\circle*{4}}
\put(117,165){\circle*{4}}
\qbezier(29,165)(73,165)(117,165)
\qbezier(29,209)(29,187)(29,165)
\qbezier(117,209)(117,187)(117,165)
\qbezier(29,209)(76,241)(117,209)
\qbezier(29,165)(76,128)(117,165)
\put(15,184){\circle*{4}}
\qbezier(29,209)(22,197)(15,184)
\qbezier(15,184)(22,175)(29,165)
\qbezier(57,184)(64,175)(72,165)
\qbezier(115,209)(72,209)(29,209)
\put(100,183){\circle*{4}}
\qbezier(117,166)(108,175)(100,183)
\qbezier(100,183)(109,196)(118,208)
\put(58,183){\circle*{4}}
\qbezier(57,184)(64,196)(72,208)
\put(58,183){\circle*{4}}
\put(85,198){\circle*{4}}
\qbezier(58,183)(71,191)(85,198)
\put(100,183){\circle*{4}}
\qbezier(100,183)(92,191)(85,198)
\qbezier(58,183)(79,183)(100,183)
\put(92,199){$u$}
\put(1,184){$v$}
\put(165,185){\circle*{4}}
\qbezier(191,157)(178,171)(165,185)
\qbezier(274,184)(260,171)(247,157)
\qbezier(191,211)(178,198)(165,185)
\qbezier(191,211)(175,182)(191,157)
\put(149,184){$u$}
\put(279,183){$v$}
\qbezier(247,211)(268,182)(247,157)
\put(208,135){$J_2$}
\put(141,52){\circle*{4}}
\put(176,52){\circle*{4}}
\qbezier(141,52)(158,52)(176,52)
\put(194,74){\circle*{4}}
\qbezier(176,52)(185,63)(194,74)
\put(194,33){\circle*{4}}
\qbezier(176,52)(185,43)(194,33)
\qbezier(194,74)(194,54)(194,33)
\put(236,74){\circle*{4}}
\qbezier(194,74)(215,74)(236,74)
\put(236,33){\circle*{4}}
\qbezier(194,33)(215,33)(236,33)
\qbezier(236,74)(236,54)(236,33)
\put(253,52){\circle*{4}}
\qbezier(236,74)(244,63)(253,52)
\qbezier(253,52)(244,43)(236,33)
\qbezier(194,74)(215,54)(236,33)
\qbezier(236,74)(215,54)(194,33)
\put(284,52){\circle*{4}}
\qbezier(253,52)(268,52)(284,52)
\put(215,88){\circle*{4}}
\qbezier(141,52)(151,109)(215,88)
\qbezier(284,52)(280,111)(215,88)
\qbezier(215,88)(172,87)(176,52)
\qbezier(215,88)(267,84)(253,52)
\put(202,12){$B_1$}
\put(128,51){$u$}
\put(296,53){$v$}
\put(362,217){\circle*{4}}
\put(362,168){\circle*{4}}
\qbezier(362,217)(362,193)(362,168)
\put(416,217){\circle*{4}}
\qbezier(362,217)(389,217)(416,217)
\put(416,168){\circle*{4}}
\qbezier(362,168)(389,168)(416,168)
\qbezier(416,217)(416,193)(416,168)
\put(389,218){\circle*{4}}
\put(362,191){\circle*{4}}
\qbezier(389,218)(375,205)(362,191)
\put(390,167){\circle*{4}}
\qbezier(362,191)(376,179)(390,167)
\put(416,193){\circle*{4}}
\qbezier(389,218)(402,206)(416,193)
\qbezier(416,193)(403,180)(390,167)
\put(341,192){\circle*{4}}
\qbezier(362,217)(351,205)(341,192)
\qbezier(341,192)(351,180)(362,168)
\qbezier(362,217)(348,192)(362,168)
\qbezier(341,192)(331,245)(416,217)
\put(415,150){\circle*{4}}
\qbezier(415,150)(454,183)(416,217)
\qbezier(341,192)(331,147)(415,150)
\put(418,166){$u$}
\put(420,144){$v$}
\put(375,136){$J_3$}
\put(100,112){Operation 1}
\put(212,112){$\Downarrow$}
\end{picture}
\end{center}
\begin{center}
\setlength{\unitlength}{0.23mm}
\begin{picture}(500,233)
\put(248,91){\circle*{4}}
\put(248,38){\circle*{4}}
\qbezier(248,91)(248,65)(248,38)
\put(221,92){\circle*{4}}
\put(192,63){\circle*{4}}
\qbezier(221,92)(206,78)(192,63)
\put(222,37){\circle*{4}}
\qbezier(192,63)(207,50)(222,37)
\put(249,64){\circle*{4}}
\qbezier(221,92)(235,78)(249,64)
\qbezier(249,64)(235,51)(222,37)
\put(192,91){\circle*{4}}
\put(192,38){\circle*{4}}
\qbezier(192,38)(220,38)(248,38)
\put(248,38){\circle*{4}}
\qbezier(192,91)(192,65)(192,38)
\put(66,12){$J_6$}
\put(30,90){\circle*{4}}
\put(118,90){\circle*{4}}
\put(73,89){\circle*{4}}
\put(73,46){\circle*{4}}
\qbezier(73,89)(73,68)(73,46)
\put(30,46){\circle*{4}}
\put(118,46){\circle*{4}}
\qbezier(30,46)(74,46)(118,46)
\qbezier(30,90)(30,68)(30,46)
\qbezier(118,90)(118,68)(118,46)
\qbezier(30,90)(77,122)(118,90)
\qbezier(30,46)(77,9)(118,46)
\put(16,65){\circle*{4}}
\qbezier(30,90)(23,78)(16,65)
\qbezier(16,65)(23,56)(30,46)
\qbezier(116,90)(73,90)(30,90)
\put(59,64){\circle*{4}}
\put(59,64){\circle*{4}}
\put(101,64){\circle*{4}}
\put(1,64){$u$}
\put(46,64){$v$}
\put(166,66){\circle*{4}}
\qbezier(192,38)(179,52)(166,66)
\qbezier(192,91)(179,79)(166,66)
\qbezier(192,91)(176,63)(192,38)
\put(150,65){$u$}
\put(257,37){$v$}
\put(209,16){$J_7$}
\put(61,217){\circle*{4}}
\put(61,168){\circle*{4}}
\qbezier(61,217)(61,193)(61,168)
\put(115,217){\circle*{4}}
\qbezier(61,217)(88,217)(115,217)
\put(115,168){\circle*{4}}
\qbezier(61,168)(88,168)(115,168)
\qbezier(115,217)(115,193)(115,168)
\put(88,218){\circle*{4}}
\put(61,191){\circle*{4}}
\qbezier(88,218)(74,205)(61,191)
\put(89,167){\circle*{4}}
\qbezier(61,191)(75,179)(89,167)
\put(115,193){\circle*{4}}
\qbezier(88,218)(101,206)(115,193)
\qbezier(115,193)(102,180)(89,167)
\put(44,192){\circle*{4}}
\qbezier(61,217)(52,205)(44,192)
\qbezier(44,192)(52,180)(61,168)
\put(14,217){$u$}
\put(14,167){$v$}
\put(64,136){$J_4$}\put(330,136){$J_5$}
\put(289,60){$\Rightarrow$}\put(240,-10){Operation 4}
\put(86,64){$w$}
\put(257,91){$w$}
\qbezier(59,64)(66,77)(73,89)
\qbezier(59,64)(66,55)(73,46)
\qbezier(101,64)(109,77)(118,90)
\qbezier(101,64)(109,55)(118,46)
\qbezier(192,91)(220,91)(248,91)
\put(27,217){\circle*{4}}
\qbezier(27,217)(35,205)(44,192)
\put(26,168){\circle*{4}}
\qbezier(44,192)(35,180)(26,168)
\qbezier(27,217)(44,217)(61,217)
\qbezier(26,168)(43,168)(61,168)
\put(126,216){$w$}
\put(124,166){$x$}
\put(147,191){$\Rightarrow$}\put(107,122){Operation 2}
\put(190,216){\circle*{4}}
\put(240,216){\circle*{4}}
\qbezier(190,216)(215,216)(240,216)
\put(215,191){\circle*{4}}
\qbezier(190,216)(202,204)(215,191)
\qbezier(240,216)(227,204)(215,191)
\put(192,168){\circle*{4}}
\qbezier(215,191)(203,180)(192,168)
\put(238,168){\circle*{4}}
\qbezier(192,168)(215,168)(238,168)
\qbezier(215,191)(226,180)(238,168)
\put(179,215){$u$}
\put(177,170){$v$}
\put(249,217){$w$}
\put(247,168){$x$}
\put(339,42){\circle*{4}}
\put(370,86){\circle*{4}}
\qbezier(339,42)(354,64)(370,86)
\put(403,42){\circle*{4}}
\qbezier(370,86)(386,64)(403,42)
\qbezier(339,42)(371,42)(403,42)
\put(326,43){$u$}
\put(405,43){$v$}
\put(382,87){$w$}
\put(311,218){\circle*{4}}
\put(311,170){\circle*{4}}
\qbezier(311,218)(311,194)(311,170)
\put(295,192){\circle*{4}}
\qbezier(311,218)(303,205)(295,192)
\qbezier(295,192)(303,181)(311,170)
\put(359,218){\circle*{4}}
\qbezier(311,218)(335,218)(359,218)
\put(359,170){\circle*{4}}
\qbezier(311,170)(335,170)(359,170)
\qbezier(359,218)(359,194)(359,170)
\qbezier(311,218)(335,194)(359,170)
\qbezier(359,218)(335,194)(311,170)
\put(376,193){\circle*{4}}
\qbezier(359,218)(367,206)(376,193)
\qbezier(376,193)(367,182)(359,170)
\put(335,194){\circle*{4}}
\put(335,194){\circle*{4}}
\put(440,217){\circle*{4}}
\put(487,217){\circle*{4}}
\qbezier(440,217)(463,217)(487,217)
\put(440,170){\circle*{4}}
\qbezier(440,217)(440,194)(440,170)
\put(487,170){\circle*{4}}
\qbezier(440,170)(463,170)(487,170)
\qbezier(487,217)(487,194)(487,170)
\put(427,193){\circle*{4}}
\qbezier(440,217)(433,205)(427,193)
\qbezier(427,193)(433,182)(440,170)
\put(501,192){\circle*{4}}
\qbezier(487,217)(494,205)(501,192)
\qbezier(501,192)(494,181)(487,170)
\qbezier(440,217)(463,194)(487,170)
\qbezier(487,217)(463,194)(440,170)
\put(396,192){$\Rightarrow$}\put(360,122){Operation 3}
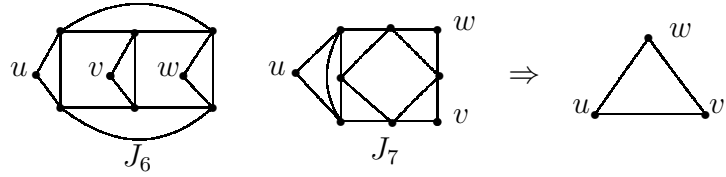
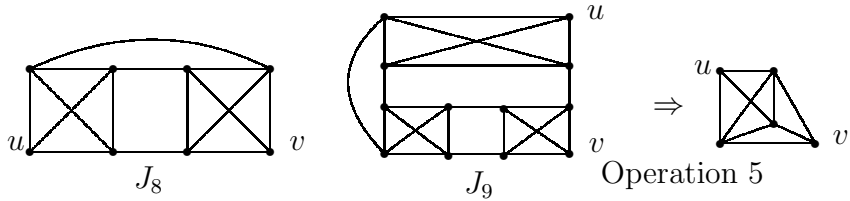
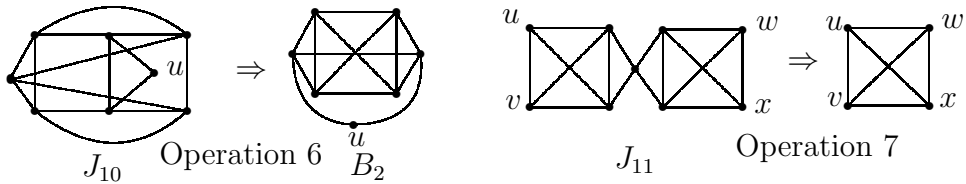
\put(292,177){$u$}
\put(374,178){$v$}
\put(422,177){$u$}
\put(500,176){$v$}
\end{picture}
\end{center}

\begin{center}
\setlength{\unitlength}{0.23mm}
\begin{picture}(571,259)
\put(358,19){$J_{11}$}
\put(23,96){\circle*{4}}
\put(111,96){\circle*{4}}
\put(66,95){\circle*{4}}
\put(66,52){\circle*{4}}
\qbezier(66,95)(66,74)(66,52)
\put(23,52){\circle*{4}}
\put(111,52){\circle*{4}}
\qbezier(23,52)(67,52)(111,52)
\qbezier(23,96)(23,74)(23,52)
\qbezier(111,96)(111,74)(111,52)
\qbezier(23,96)(70,128)(111,96)
\qbezier(23,52)(70,15)(111,52)
\put(9,71){\circle*{4}}
\qbezier(23,96)(16,84)(9,71)
\qbezier(9,71)(16,62)(23,52)
\qbezier(109,96)(66,96)(23,96)
\put(92,74){\circle*{4}}
\put(9,70){\circle*{4}}
\put(99,73){$u$}
\put(295,53){$v$}
\put(139,71){$\Rightarrow$}\put(95,21){Operation 6}\put(205,15){$B_2$}
\put(440,99){$w$}
\qbezier(92,74)(79,85)(66,95)
\qbezier(92,74)(79,63)(66,52)
\qbezier(9,70)(60,83)(111,96)
\qbezier(9,70)(60,61)(111,52)
\put(422,168){\circle*{4}}
\put(453,210){\circle*{4}}
\qbezier(422,168)(437,189)(453,210)
\put(453,179){\circle*{4}}
\qbezier(453,210)(453,195)(453,179)
\qbezier(422,168)(437,174)(453,179)
\put(407,208){$u$}
\put(487,168){$v$}
\put(114,211){\circle*{4}}
\put(114,163){\circle*{4}}
\qbezier(114,211)(114,187)(114,163)
\put(71,211){\circle*{4}}
\qbezier(114,211)(92,211)(71,211)
\put(162,211){\circle*{4}}
\qbezier(114,211)(138,211)(162,211)
\put(162,163){\circle*{4}}
\qbezier(114,163)(138,163)(162,163)
\qbezier(162,211)(162,187)(162,163)
\qbezier(114,211)(138,187)(162,163)
\qbezier(162,211)(138,187)(114,163)
\put(185,109){\circle*{4}}
\put(232,109){\circle*{4}}
\qbezier(185,109)(208,109)(232,109)
\put(185,62){\circle*{4}}
\qbezier(185,109)(185,86)(185,62)
\put(232,62){\circle*{4}}
\qbezier(185,62)(208,62)(232,62)
\qbezier(232,109)(232,86)(232,62)
\put(172,85){\circle*{4}}
\qbezier(185,109)(178,97)(172,85)
\qbezier(172,85)(178,74)(185,62)
\put(246,85){\circle*{4}}
\qbezier(232,109)(239,97)(246,85)
\qbezier(246,85)(239,74)(232,62)
\qbezier(185,109)(208,86)(232,62)
\qbezier(232,109)(208,86)(185,62)
\put(383,185){$\Rightarrow$}\put(353,145){Operation 5}
\put(10,164){$u$}
\put(173,164){$v$}
\put(204,32){$u$}
\qbezier(172,85)(209,85)(246,85)
\put(207,44){\circle*{4}}
\qbezier(172,85)(173,46)(207,44)
\qbezier(246,85)(250,45)(207,44)
\put(71,163){\circle*{4}}
\qbezier(71,211)(71,187)(71,163)
\qbezier(71,163)(92,163)(114,163)
\put(23,211){\circle*{4}}
\qbezier(23,211)(47,211)(71,211)
\put(23,163){\circle*{4}}
\qbezier(23,211)(23,187)(23,163)
\qbezier(23,163)(47,163)(71,163)
\qbezier(23,211)(47,187)(71,163)
\qbezier(71,211)(47,187)(23,163)
\qbezier(23,211)(95,245)(162,211)
\put(422,210){\circle*{4}}
\qbezier(422,210)(437,210)(453,210)
\qbezier(422,210)(437,195)(453,179)
\qbezier(422,210)(422,189)(422,168)
\put(477,168){\circle*{4}}
\qbezier(453,210)(465,189)(477,168)
\qbezier(453,179)(465,174)(477,168)
\qbezier(422,168)(449,168)(477,168)
\put(309,100){\circle*{4}}
\put(309,54){\circle*{4}}
\qbezier(309,100)(309,77)(309,54)
\put(355,100){\circle*{4}}
\qbezier(309,100)(332,100)(355,100)
\put(355,54){\circle*{4}}
\qbezier(355,100)(355,77)(355,54)
\qbezier(309,54)(332,54)(355,54)
\qbezier(309,54)(332,77)(355,100)
\qbezier(309,100)(332,77)(355,54)
\put(370,76){\circle*{4}}
\qbezier(355,100)(362,88)(370,76)
\qbezier(370,76)(362,65)(355,54)
\put(386,99){\circle*{4}}
\qbezier(370,76)(378,88)(386,99)
\put(386,54){\circle*{4}}
\qbezier(370,76)(378,65)(386,54)
\qbezier(386,99)(386,77)(386,54)
\put(432,99){\circle*{4}}
\qbezier(386,99)(409,99)(432,99)
\put(432,54){\circle*{4}}
\qbezier(432,99)(432,77)(432,54)
\qbezier(386,54)(409,54)(432,54)
\qbezier(386,99)(409,77)(432,54)
\qbezier(432,99)(409,77)(386,54)
\put(493,100){\circle*{4}}
\put(493,55){\circle*{4}}
\qbezier(493,100)(493,78)(493,55)
\put(540,100){\circle*{4}}
\qbezier(493,100)(516,100)(540,100)
\put(540,55){\circle*{4}}
\qbezier(540,100)(540,78)(540,55)
\qbezier(493,55)(516,55)(540,55)
\qbezier(493,100)(516,78)(540,55)
\qbezier(540,100)(516,78)(493,55)
\put(458,76){$\Rightarrow$}\put(428,26){Operation 7}
\put(293,103){$u$}
\put(439,51){$x$}
\put(480,99){$u$}
\put(481,55){$v$}
\put(547,100){$w$}
\put(546,54){$x$}
\put(228,189){\circle*{4}}
\put(335,189){\circle*{4}}
\qbezier(228,189)(281,189)(335,189)
\put(228,162){\circle*{4}}
\qbezier(228,189)(228,176)(228,162)
\put(335,162){\circle*{4}}
\qbezier(228,162)(281,162)(335,162)
\qbezier(335,189)(335,176)(335,162)
\put(265,189){\circle*{4}}
\put(265,161){\circle*{4}}
\qbezier(265,189)(265,175)(265,161)
\put(297,188){\circle*{4}}
\put(297,161){\circle*{4}}
\qbezier(297,188)(297,175)(297,161)
\qbezier(228,189)(246,175)(265,161)
\qbezier(265,189)(246,176)(228,162)
\qbezier(232.69,165.42)(228,162)(228,162)
\qbezier(297,188)(316,175)(335,162)
\qbezier(335,189)(316,175)(297,161)
\put(228,213){\circle*{4}}
\qbezier(228,189)(228,201)(228,213)
\put(335,213){\circle*{4}}
\qbezier(335,189)(335,201)(335,213)
\qbezier(228,213)(281,213)(335,213)
\put(228,241){\circle*{4}}
\qbezier(228,241)(228,227)(228,213)
\put(335,241){\circle*{4}}
\qbezier(228,241)(281,241)(335,241)
\qbezier(335,241)(335,227)(335,213)
\qbezier(228,241)(281,227)(335,213)
\qbezier(335,241)(281,227)(228,213)
\qbezier(228,241)(186,205)(228,162)
\put(345,240){$u$}
\put(346,163){$v$}
\put(83,142){$J_8$}
\put(274,139){$J_9$}
\put(50,14){$J_{10}$}
\end{picture}
\end{center}

\begin{center}
\setlength{\unitlength}{0.23mm}
\begin{picture}(565,259)
\put(45,148){$J_{12}$}
\put(48,55){$v$}
\put(123,222){$w$}
\put(474,236){$u$}
\put(542,234){$v$}
\put(39,222){\circle*{4}}
\put(86,222){\circle*{4}}
\qbezier(39,222)(62,222)(86,222)
\put(39,175){\circle*{4}}
\qbezier(39,222)(39,199)(39,175)
\put(86,175){\circle*{4}}
\qbezier(39,175)(62,175)(86,175)
\qbezier(86,222)(86,199)(86,175)
\put(26,198){\circle*{4}}
\qbezier(39,222)(32,210)(26,198)
\qbezier(26,198)(32,187)(39,175)
\put(100,198){\circle*{4}}
\qbezier(86,222)(93,210)(100,198)
\qbezier(100,198)(93,187)(86,175)
\qbezier(39,222)(62,199)(86,175)
\qbezier(86,222)(62,199)(39,175)
\put(140,198){$\Rightarrow$}\put(100,148){Operation 8}
\put(1,221){$u$}
\put(2,174){$v$}
\put(95,99){\circle*{4}}
\put(95,54){\circle*{4}}
\qbezier(95,99)(95,77)(95,54)
\put(141,99){\circle*{4}}
\qbezier(95,99)(118,99)(141,99)
\put(141,54){\circle*{4}}
\qbezier(141,99)(141,77)(141,54)
\qbezier(95,54)(118,54)(141,54)
\qbezier(95,54)(118,77)(141,99)
\qbezier(95,99)(118,77)(141,54)
\put(172,99){\circle*{4}}
\put(172,54){\circle*{4}}
\qbezier(172,99)(172,77)(172,54)
\put(218,99){\circle*{4}}
\qbezier(172,99)(195,99)(218,99)
\put(218,54){\circle*{4}}
\qbezier(218,99)(218,77)(218,54)
\qbezier(172,54)(195,54)(218,54)
\qbezier(172,99)(195,77)(218,54)
\qbezier(218,99)(195,77)(172,54)
\put(486,236){\circle*{4}}
\put(486,191){\circle*{4}}
\qbezier(486,236)(486,214)(486,191)
\put(533,236){\circle*{4}}
\qbezier(486,236)(509,236)(533,236)
\put(533,191){\circle*{4}}
\qbezier(533,236)(533,214)(533,191)
\qbezier(486,191)(509,191)(533,191)
\qbezier(486,236)(509,214)(533,191)
\qbezier(533,236)(509,214)(486,191)
\put(444,201){$\Rightarrow$}\put(400,140){Operation 9}
\put(47,98){$u$}
\put(124,175){$x$}
\put(170,222){$u$}
\put(170,171){$v$}
\put(215,222){$w$}
\put(215,172){$x$}
\put(289,189){\circle*{4}}
\put(396,189){\circle*{4}}
\qbezier(289,189)(342,189)(396,189)
\put(289,162){\circle*{4}}
\qbezier(289,189)(289,176)(289,162)
\put(396,162){\circle*{4}}
\qbezier(289,162)(342,162)(396,162)
\qbezier(396,189)(396,176)(396,162)
\put(326,189){\circle*{4}}
\put(326,161){\circle*{4}}
\qbezier(326,189)(326,175)(326,161)
\put(358,188){\circle*{4}}
\put(358,161){\circle*{4}}
\qbezier(358,188)(358,175)(358,161)
\qbezier(289,189)(307,175)(326,161)
\qbezier(358,188)(377,175)(396,162)
\qbezier(396,189)(377,175)(358,161)
\put(289,213){\circle*{4}}
\qbezier(289,189)(289,201)(289,213)
\put(396,213){\circle*{4}}
\qbezier(396,189)(396,201)(396,213)
\qbezier(289,213)(342,213)(396,213)
\put(289,241){\circle*{4}}
\qbezier(289,241)(289,227)(289,213)
\put(396,241){\circle*{4}}
\qbezier(289,241)(342,241)(396,241)
\qbezier(396,241)(396,227)(396,213)
\qbezier(289,241)(342,227)(396,213)
\qbezier(396,241)(342,227)(289,213)
\put(275,240){$u$}
\put(405,240){$v$}
\put(327,145){$J_{13}$}
\put(137,25){$J_{14}$}
\qbezier(326,189)(307,176)(289,162)
\put(13,221){\circle*{4}}
\qbezier(26,198)(19,210)(13,221)
\put(15,175){\circle*{4}}
\qbezier(26,198)(20,187)(15,175)
\put(113,223){\circle*{4}}
\qbezier(100,198)(106,211)(113,223)
\put(114,176){\circle*{4}}
\qbezier(100,198)(107,187)(114,176)
\put(194,196){\circle*{4}}
\put(208,223){\circle*{4}}
\qbezier(194,196)(201,210)(208,223)
\put(181,223){\circle*{4}}
\qbezier(181,223)(187,210)(194,196)
\put(182,173){\circle*{4}}
\qbezier(194,196)(188,185)(182,173)
\put(207,172){\circle*{4}}
\qbezier(194,196)(200,184)(207,172)
\put(265,162){\circle*{4}}
\qbezier(265,162)(277,162)(289,162)
\put(422,162){\circle*{4}}
\qbezier(396,162)(409,162)(422,162)
\put(262,174){$w$}
\put(419,175){$x$}
\put(486,160){\circle*{4}}
\qbezier(486,191)(486,176)(486,160)
\put(533,159){\circle*{4}}
\qbezier(533,191)(533,175)(533,159)
\put(472,162){$w$}
\put(544,160){$x$}
\qbezier(141,99)(156,99)(172,99)
\qbezier(141,54)(156,54)(172,54)
\put(63,99){\circle*{4}}
\qbezier(63,99)(79,99)(95,99)
\put(63,54){\circle*{4}}
\qbezier(63,54)(79,54)(95,54)
\put(253,99){\circle*{4}}
\qbezier(218,99)(235,99)(253,99)
\put(253,54){\circle*{4}}
\qbezier(218,54)(235,54)(253,54)
\put(266,99){$w$}
\put(269,53){$x$}
\put(301,76){$\Rightarrow$}\put(261,26){Operation 10}
\put(348,93){\circle*{4}}
\put(395,93){\circle*{4}}
\qbezier(348,93)(371,93)(395,93)
\put(349,56){\circle*{4}}
\put(397,56){\circle*{4}}
\qbezier(349,56)(373,56)(397,56)
\put(332,95){$u$}
\put(333,58){$v$}
\put(405,93){$w$}
\put(408,57){$x$}
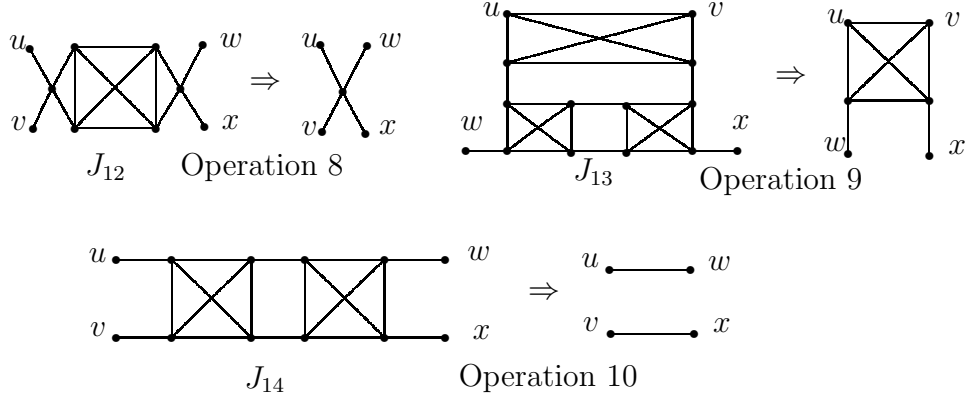
\put(66,-9){Figure 4. Forbidden subgraphs $J_i$ for $i\in\{1,2,\cdots,14\}$.}
\end{picture}
\end{center}

Figure 4 give some replacement operations of graphs. For each operation, we replace $J_i$ by a smaller subgraph with identical labeled vertices. It is easy to check that the resulting graph is also 4-regular and claw-free. For example, for $i=1,2,3$, we can replace each of $J_i$ by $B_1$ by identifying the two vertices $u$, $v$ and keeping the neighborhood of the two vertices in $G'=(V,E(G)\setminus E(J_i))$.


\begin{lemma}\label{forbsub}
For each $i\in\{1,2,\cdots,14\}$, we have the following statement.

(i) For $i\in\{1,2,3,4,5,8,9,10,12,13\}$, $G$ does not contain $J_i$ as a subgraph.

(ii) For $i\in \{6,7,11,14\}$, $G$ does not contain $J_i$ as an induced subgraph.

\end{lemma}

\begin{proof}
We only prove $G$ does not contain $J_1$ as a subgraph or $J_{14}$ as an induced subgraph. The proof for the other forbidden subgraphs is completely similar and we omit it here.

If $G$ contains $J_1$ as a subgraph, then we can obtain $G'$ from $G$ by Operation 1. Then $\gamma_p(G')\leq \frac {n'+1}5$, where $n'$ is the order of $G'$. Suppose that $M'$ is a $\gamma_p(G')$-set. It is easy to check that $M'$ contains at least one saturated vertex of $B_1$. If $M'$ contains at least two saturated vertices of $B_1$, say $w_1,w_2$, then $M'\setminus\{w_1,w_2\}$ together with any two saturated vertices of $V(J_1)$ will form a PDS of $G$. If $M'$ contains exactly one saturated vertex of $B_1$, say $w_1$, then $w_1$ is in the copy of $K_4$ in $B_1$. Hence $u$ or $v$ must receive message from $M'\setminus V(B_1)$. Therefore, $M'\setminus\{w_1\}$ together with any saturated vertex nonadjacent to $u$ or $v$ in $J_1$ form a PDS of $G$. Thus, $|M|\leq|M'|\leq \frac {n'+1}5\leq \frac {n+1}5$. A contradiction.

If $G$ contains $J_{14}$ as an induced subgraph, then we can obtain $G'$ from $G$ by operation 10. Suppose first that $G'$ contains two components, where their order are $n_1$ and $n_2$, respectively. 
Let $M'$ be a $\gamma_p(G')$-set. If $M'\cap \{u, v, x, w\}\neq\emptyset$, say $u$, then we add a vertex $u'$ with $d_{J_{14}}(u, u')=3$ to $M_1\cup M_2$. The resulting vertex set $M$ is a PDS of $G$. Then $|M|\leq \frac {n_1+1}5+\frac{n_2+1}5+1\leq \frac {n-6+2}5+1\leq \frac {n+1}5$, a contradiction.
Suppose that $M'\cap \{u, v, x, w\}=\emptyset$. By the symmetry if neither of $\{u,w\}$ can send message to each other, then $u$ receives message from $N_{G'}(u)\setminus \{w\}$ which induces a copy of $K_4$ in  $G'$. Therefore $u$ can send message to $w$, a contradiction. Thus at least one vertex of $\{u, v, x, w\}$ can send message to its neighbor in $L_2$ of $G$, say $u$.  $M=M'\cup\{u'\}$ is a PDS of $G$, where $d_{J_{14}}(u, u')=3$. Then $|M|\leq \frac {n_1+1}5+\frac{n_2+1}5+1\leq \frac {n-6+2}5+1\leq \frac {n+1}5$, a contradiction. If $G'$ is the connected graph, then the discussion is similar and slightly easy than above. We omit it here.
\end{proof}

Let $A=\{A_1,A_2,A_3\}$ be the collection of graphs in Figure 5. The three subgraphs need to pay our more attention.

\begin{center}
\setlength{\unitlength}{0.15mm}
\begin{picture}(465,130)
\put(342,124){\circle*{6}}
\put(342,23){\circle*{6}}
\put(454,124){\circle*{6}}
\put(454,23){\circle*{6}}
\qbezier(342,124)(342,74)(342,23)
\qbezier(342,124)(398,124)(454,124)
\qbezier(454,124)(454,74)(454,23)
\qbezier(342,23)(398,23)(454,23)
\put(10,91){\circle*{6}}
\put(90,91){\circle*{6}}
\put(10,17){\circle*{6}}
\put(90,17){\circle*{6}}
\qbezier(10,91)(50,91)(90,91)
\qbezier(10,17)(50,17)(90,17)
\qbezier(10,17)(50,54)(90,91)
\qbezier(10,91)(50,54)(90,17)
\qbezier(10,91)(10,54)(10,17)
\qbezier(90,91)(90,54)(90,17)
\put(44,142){\circle*{6}}
\qbezier(44,142)(67,117)(90,91)
\qbezier(44,142)(27,117)(10,91)
\put(205,129){\circle*{6}}
\put(155,17){\circle*{6}}
\qbezier(205,129)(180,73)(155,17)
\put(275,21){\circle*{6}}
\qbezier(205,129)(240,75)(275,21)
\qbezier(155,17)(215,19)(275,21)
\put(182,74){\circle*{6}}
\put(211,20){\circle*{6}}
\qbezier(182,74)(196,47)(211,20)
\put(241,74){\circle*{6}}
\qbezier(182,74)(211,74)(241,74)
\qbezier(241,74)(226,47)(211,20)
\put(394,123){\circle*{6}}
\put(344,73){\circle*{6}}
\qbezier(394,123)(369,98)(344,73)
\put(400,25){\circle*{6}}
\qbezier(344,73)(372,49)(400,25)
\put(454,73){\circle*{6}}
\qbezier(394,123)(424,98)(454,73)
\qbezier(454,73)(427,49)(400,25)
\put(40,-20){$A_1$}\put(195,-20){$A_2$}\put(382,-20){$A_3$}
\put(80,-60){Figure 5. $A_i$ for $i\in\{1,2,3\}$.}
\end{picture}
\end{center}

\vskip 2em

For $i,j\in\{1,2,3\}$, let $A_i\circ A_j$ be the graph obtained from $A_i$ and $A_j$ by identifying exactly one vertex of degree two of them. We call the sharing vertex a \emph{focal vertex}. Note that $A_1\circ A_1\cong J_{11}$. Then $G$ contains no induced subgraph isomorphic to $A_1\circ A_1$.

\begin{lemma}\label{le3}
Let $H_1\cong A_i$ and $H_2\cong A_j$ be two subgraphs of $G$, where $i\in\{1,2,3\}$ and $j\in\{2,3\}$. If there is a vertex $x$ in $H_1\cap H_2$, then $H_1\cup H_2\cong A_i\circ A_j$ and $x$ is a focal vertex of $A_i\circ A_j$.
\end{lemma}

\begin{proof}
If $H_1\cong A_1$ and $H_2\cong A_j$ for $j\in\{2,3\}$, then the statement is obvious.
Now we consider the case for $H_1\cong A_i$ and $H_2\cong A_j$, where $i,j\in\{2,3\}$. Let $C_1$ and $C_2$ be two saturated cycles of $H_1$ and $H_2$ respectively. We partition our proof into the following claims.

\noindent
{\bf Claim 1.} $|V(C_1)\cap V(C_2)|\le 1$.

\begin{proof}
 By the claw-freeness of $G$ and the configuration of $A_2$ and $A_3$, it is easy to check that $|V(C_1)\cap V(C_2)|\leq 2$. If the statement is false, then we suppose that $|V(C_1)\cap V(C_2)|=2$.

If $C_1:=v_1v_2v_3v_1$ and $C_2:=u_1v_2v_3u_1$ are both saturated triangles, then by the configuration of $A_2$, there are two vertices $w_1$ and $w_2$ in $G$ such that all $w_1v_1$, $w_1v_2$, $w_2v_1$ and $w_2v_3$ are in $E(G)$. Since $C_2$ is the saturated triangle and by the 4-regular of $G$, $w_1u_1, w_2u_1\in E(G)$. By the claw-freeness of $G$, $w_1w_2\in E(G)$. Thus, $G\cong I_2$. This contradicts Lemma \ref{counterex}.

If $C_1:=v_1v_2v_3v_1$ is a saturated triangle and $C_2:=v_1v_2u_1u_2v_1$ is a saturated quadrilateral, then by the configuration of $A_2$, $u_1v_3, u_2v_3\in E(G)$. But it contradicts the configuration of $A_3$.

Finally we consider $C_1$ and $C_2$ are two saturated quadrilaterals. Note that now $C_1$ and $C_2$ share two adjacent vertices. Let $C_1:=wvv_1w'w$ and $C_2:=wuu_1w'w$. By the configuration of $A_3$, we have $uv, u_1v_1\in E(G)$.
Then there are three vertices $w_1,w_2,w_3$ such that $w_1u$, $w_1u_1$, $w_2w$, $w_2w'$, $w_3v$, $w_3v_1$ are in $E(G)$, here $w_1=w_3$ is allowed.

Suppose first $w_1\ne w_3$.
Since $G$ is claw-free, $G[w_1,w_2,w_3]$ is either a clique or contains at most one edge. If $\{w_1, w_2, w_3\}$ induces a clique, then $G\cong I_5$. This contradicts Lemma \ref{counterex}. Now without loss of generality suppose $w_1w_2\in E(G)$. Since $G$ is claw-free and 4-regular, then $w_1$ and $w_2$ share a neighbor $w_4$. Now $G[V(H_1\cup H_{2})\cup \{w_4\}]\cong J_1$. This contradicts Lemma \ref{forbsub}.

Now we consider $\{w_1, w_2, w_3\}$ is an independent set. Now $G[V(H_1\cup H_{2})]\cong J_6$. This contradicts Lemma \ref{forbsub} as well.

If $w_1=w_3$, then $G[V(H_1\cup H_2)]\cong J_{10}$. A contradiction.
\end{proof}

\noindent
{\bf Claim 2.} $|V(C_1)\cap V(C_2)|=0$.

\begin{proof}
Otherwise, suppose that there is a vertex $w\in V(C_1\cap C_2)$. Then we consider the three cases as follows.

{\bf Case 1.} $C_1:=v_1v_2wv_1$ and $C_2:=u_1u_2wu_1$ are two saturated triangles.

By the configuration of  $A_2$, let $u_1v_1, u_2v_2\in E(G)$ and there is another vertex $w_1\in V(H_1)$ adjacent to $v_1, v_2$. If $w_1$ is also the common neighbor of $u_1$ and $u_2$, then $G\cong I_2$, a contradiction.

Otherwise, there is the vertex $w_2\in V(H_2)$ (other than $w$) adjacent to both $u_1$ and $u_2$. Then $G[V(H_1\cup H_2)]\cong J_5$, a contradiction.

{\bf Case 2.} $C_1:=u_1u_2wu_1$ is a saturated triangle and $C_2:=v_1v_3v_2wv_1$ is a saturated quadrilateral.

By the configuration of $A_2$ and $A_3$, without loss of generality say $u_1v_1, u_2v_2\in E(G)$ and there is a vertex $w_1(\ne v_3)$ adjacent to $u_1, u_2$. If $w_1$ is also the common neighbor of $v_1, v_3$ (or $v_2,v_3$), then $G[\{u_1, u_2, w\}]$ and $G[\{w_1, u_1, v_1\}]$ (or $G[\{w_1, u_2, v_2\}]$) are two saturated triangles with common vertex $u_1$(or $u_2$). By the discussion of Case 1 above, we have $\gamma_{p}(G)\le\frac{n+1}{5}$, a contradiction. Therefore there are three distinct vertices $w_1, w_2, w_3$ with all $w_1u_1$, $w_1u_2$, $w_2v_2$, $w_2v_3$, $w_3v_3$, $w_3v_1$ in $E(G)$. If $\{w_1, w_2, w_3\}$ induces a clique, then $G\cong I_3$. This contradicts Lemma \ref{counterex}.

Since $G$ is claw-free, $G[w_1,w_2,w_3]$ contains at most one edge. Suppose first that $w_1w_2$ or $w_1w_3$ is in $E(G)$. Then $w_1$ and $w_2$ (or $w_3$) share a common neighbor, say $w_4$. Hence $G[V(H_1\cup H_2)\cup\{w_4\}]$ contains a subgraph isomorphic to $J_3$. A contradiction. If $w_2w_3\in E(G)$, then $w_2$ and $w_3$ also share a common neighbor, say $w_5$. Now $G[V(H_1\cup H_2)\cup\{w_5\}]$ contains a subgraph isomorphic to $J_2$. A contradiction.

If $\{w_1, w_2, w_3\}$ is an independent set, then $G[V(H_1\cup H_2)]\cong J_7$. A contradiction.

 {\bf Case 3.} $C_1:=v_1v_3wv_2 v_1$ and $C_2:=u_1u_3wu_2u_1$ are two saturated quadrilaterals.

 Without loss of generality we can suppose that $u_2v_2, u_3v_3\in E(G)$. By the configuration of $A_3$, $v_1$ and $v_3$ share a neighbor $w_1$, $v_1$ and $v_2$ share a neighbor $w_2$. If $w_1$ is also the common neighbor of $u_1$ and $u_3$, then $G[\{w_1, u_3, v_3\}]$ is a saturated triangle of some $A_2$ intersecting  the saturated quadrilateral $G[\{v_1, v_2, v_3, w\}]$. By the discussion of Case 2 above, we are done. If $w_1$ is the common neighbor of $u_1$ and $u_2$, then there is another saturated quadrilateral $G[\{w_1, v_1, v_2, u_2\}]$ intersecting $G[\{v_1, v_2, v_3, w\}]$ at $v_1$ and $v_2$. The case has been done in Claim 1.

We suppose that there are two vertices $w_3, w_4$ with all of $w_3u_1, w_3u_3$, $w_4u_1, w_4u_2$ in $E(G)$. Base on the discussion above, we can suppose that all of $w_1$, $w_2,w_3$ and $w_4$ are distinct. By symmetry, if $w_1 w_2\in E(G)$, then $G[\{w_1, w_2, v_1\}]$ is a saturated triangle intersecting $G[\{v_1, v_2, v_3, w\}]$ at $v_1$. By the discussion of Case 2 above, we are done. If $w_3w_4\in E(G)$, the similar discussion is valid as well.

Now we consider the case $w_1w_2, w_3w_4\notin E(G)$. Then $G[V(H_1\cup H_2)]$ contains a subgraph isomorphic to $J_4$. A contradiction.
\end{proof}

\noindent
{\bf Claim 3.} $x\notin V(C_1\cup C_2)$.

\begin{proof}
Otherwise, without loss of generality say that $x\in V(C_1)$. By Claim 2, $C_1$ and $C_2$ are disjoint. First suppose that both $C_1$ and $C_2$ are saturated triangles. We can suppose that $C_1:=xv_1v_2x$, $C_2:=u_1u_2u_3u_1$ and $xu_1, xu_2\in E(G)$. Based on the configuration of $A_2$, all of $u_1v_1, u_2v_2$, $v_1u_3, v_2u_3$ are in $E(G)$. Then $G\cong I_2$. This contradicts Lemma \ref{counterex}.

Now suppose that $C_1:=xv_1v_2x$ is a saturated triangle and $C_2:=u_1u_2u_3u_4u_1$ is a saturated quadrilateral. Without loss of generality say $xu_1, xu_2\in E(G)$. Consider the configuration of $A_2$ and $A_3$, we have all of $u_1v_1, u_2v_2$ $u_3v_2$ and $u_4v_1$ in $E(G)$. But this contradicts that $v_1$ and $v_2$ share a neighbor other than $x$.

Finally suppose that $C_1:=v_1xv_2v_3v_1$ and $C_2:=u_1u_2u_3u_4u_1$ are two saturated quadrilaterals. Without loss of generality say $u_1x, u_2x\in E(G)$. Based on the configuration of $A_3$, all of $u_1v_1, u_2v_2$, $v_1u_3$ and $v_2u_4$ are in $E(G)$. Since $G$ is claw-free, $v_3u_3, v_3u_4\in E(G)$. Therefore, $G\cong I_4$, contradicting Lemma \ref{counterex}.
\end{proof}
From Claims 2 and 3, $x$ is exactly the focal vertex of $H_1\cup H_2$. This completes the proof.
\end{proof}




\begin{lemma}\label{le4}
Let $H$ isomorphic to $A_i$ for $i=\{2, 3\}$ be a subgraph of $G$, then $H$ is also an induced subgraph of $G$.
\end{lemma}

\begin{proof}
If $G$ contains a subgraph $H$ isomorphic to $A_2$, then let $C:=v_1v_2v_3v_1$ be the saturated cycle of $H$ and $w_1,w_2,w_3$ are three vertices of $H$ with $w_i$ is adjacent to both of $v_i$ and $v_{i+1}$, where $i\in\{1,2,3\}$ and the index of $v_i$ taken modulo $3$. If $w_1w_2\in E(G)$, then $w_1$ and $w_2$ share a neighbor other than $v_2$, say $w_4$. But now $G[V(H)\cup\{w_4\}]$ contains a subgraph isomorphic to $J_5$, a contradiction.

We now consider the case for $H$ isomorphic to $A_3$. Let $C:=u_1u_2u_3u_4u_1$ be the saturated cycle of $H$. For $i=1,2,3,4$, let $v_i$ be the vertex of degree $2$ adjacent to both of $u_i$ and $u_{i+1}$ in $H$, here the index of $v_i$ are taken modulo $4$.

First suppose that there is an edge $v_iv_{i+1}\in E(G)$ for some $i\in\{1,2,3,4\}$, without loss of generality say $v_1v_2$.

If $v_1$ is adjacent to another vertex $w$ different from $v_3$ and $v_4$, then by the claw-freeness of $G$ and since $G$ is 4-regular, $wv_2\in E(H)$. Therefore $G[V(H)\cup\{w\}]$ contains a subgraph isomorphic to $J_7$, a contradiction. Similarly if $v_1$ is adjacent to $v_4$, we can obtain a contradiction as well. If $v_1$ is adjacent to $v_3$, then $v_2v_3\in E(H)$. Now $\{u_1,u_2,u_3,v_1,v_2,v_3\}$ induces a subgraph $H_1$ isomorphic to $A_2$. Note that the subgraph and $H$ share five common vertices. But this contradicts Lemma \ref{le3}. If $v_1v_4\in E(G)$, then since $G$ is claw-free, $v_3v_4, v_2v_3\in E(H)$. Hence $G\cong I_4$. This contradicts Lemma \ref{counterex} as well. It is similar to the case for $v_1v_4$ and $v_2v_3$ are edges of $H$.

Now none of $v_iv_{i+1}$ is an edge of $G$. If $v_1v_3\in E(G)$, then $v_1$ and $v_3$ share a neighbor, say $w$. Hence $\{v_1,u_2,u_3,v_3,u_1,v_2,u_4,w\}$ induces an $A_3$ with saturated cycle $C':v_1u_2u_3v_3v_1$. The cycle $C'$ intersects $C$ at two vertices. This contradicts Lemma \ref{le3}. Similarly, we have $v_2v_4\notin E(G)$. This means that $H$ is an induced subgraph of $G$.

This completes the proof.
\end{proof}

\begin{center}
\setlength{\unitlength}{0.25mm}
\begin{picture}(345,124)
\put(292,104){\circle*{6}}
\put(292,50){\circle*{6}}
\qbezier(292,104)(292,77)(292,50)
\put(345,104){\circle*{6}}
\qbezier(292,104)(318,104)(345,104)
\put(345,50){\circle*{6}}
\qbezier(292,50)(318,50)(345,50)
\qbezier(345,104)(345,77)(345,50)
\qbezier(292,104)(318,77)(345,50)
\qbezier(345,104)(318,77)(292,50)
\put(0,76){\circle*{6}}
\put(27,103){\circle*{6}}
\put(27,49){\circle*{6}}
\put(80,103){\circle*{6}}
\put(80,49){\circle*{6}}
\qbezier(27,103)(27,76)(27,49)
\qbezier(27,103)(53,103)(80,103)
\qbezier(27,49)(53,49)(80,49)
\qbezier(80,103)(80,76)(80,49)
\qbezier(27,103)(53,76)(80,49)
\qbezier(80,103)(53,76)(27,49)
\qbezier(0,76)(13,90)(27,103)
\qbezier(0,76)(13,63)(27,49)
\put(110,73){\circle*{6}}
\qbezier(80,103)(95,88)(110,73)
\qbezier(110,73)(95,61)(80,49)
\put(194,104){\circle*{6}}
\put(194,50){\circle*{6}}
\put(247,104){\circle*{6}}
\put(247,50){\circle*{6}}
\qbezier(194,104)(194,77)(194,50)
\qbezier(194,104)(220,104)(247,104)
\qbezier(194,50)(220,50)(247,50)
\qbezier(247,104)(247,77)(247,50)
\qbezier(194,104)(220,77)(247,50)
\qbezier(247,104)(220,77)(194,50)
\put(167,77){\circle*{6}}
\qbezier(167,77)(180,91)(194,104)
\put(167,77){\circle*{6}}
\qbezier(167,77)(180,64)(194,50)
\qbezier(247,104)(269,104)(292,104)
\qbezier(247,50)(269,50)(292,50)
\put(43,28){$B_3$}
\put(260,28){$B_4$}
\put(85,-9){Figure 6. $B_3$ and $B_4$.}
\end{picture}
\end{center}

Let $B_3$ be the graph obtained from a $K_5$ by splitting its one vertex into two vertices of degree two. Let $B_4$ be the graph obtained from a $L_2$ by adding a new vertex and linking the vertex to two adjacent vertices of degree three in $L_2$, see Figure 6.
The following statement is easy to check and we omit its detailed proof.
\begin{lemma}\label{le2}
Let $H_1$ and $H_2$ be two subgraphs of $G$ sharing at least one vertex.

(i) If $H_1\cong A_1$, $H_2\cong A_1$, then $H_1\cup H_2\cong A_1\circ A_1$, $B_3$ or $K_5-e$.

(ii) If $H_1\cong A_1$, $H_2\cong L_2$, then $H_1\cup H_2\cong B_4$.
\end{lemma}

Note that $G\ncong I_7$. If $G$ contains a induced subgraph $H$ isomorphic to $A_1\circ A_1$, then there is exactly one edge linking the two copies of $K_4$ of $H$.

\begin{lemma}\label{le61}
If $G$ contains a subgraph $H$ isomorphic to $B_3$, then $G$ contains a subgraph $H'$ isomorphic to $B_1$ or $B_2$ (see Figure 4) such that $H\subset H'$.
\end{lemma}

\begin{proof}
Let $u_1$ and $u_2$ be the two vertices of degree 2 in $H$. By the claw-freeness of $G$, if $u_1u_2\in E(G)$, then $u_1$ and $u_2$ share a common neighbor, say $w$, in $G$. Since $G$ is 4-regular and claw-free, $V(H)\cup\{w\}$ induces a $B_2$.

Now suppose that $u_1u_2\notin E(G)$. If $u_1$ and $u_2$ have no common vertex, then $G[V(H)\cup N(u_1)\cup N(u_2)]$ contains a subgraph isomorphic to $J_{12}$. A contradiction.

Suppose that $u_1$ and $u_2$ have at least one common neighbor. Since $G$ is claw-free and 4-regular, $u_1$ and $u_2$ share exactly one common vertex, say $w$. (If $w'$ is the common neighbor of $u_1$ and $u_2$ other than $w$, then $\{u_1,u_2,w\}$ together with the neighbor of $w$ (other than $w'$) induces a claw). Let $w_1$ and $w_2$ be the fourth neighbors of $u_1$ and $u_2$, respectively. Note that $w_1w, w_2w\in E(G)$. The graph induced by $V(H)\cup \{w, w_1,w_2\}$ is isomorphic to $B_1$.

This completes the proof.
\end{proof}

\begin{lemma}\label{le62}
If $G$ contains a subgraph $H$ isomorphic to $L_2$, then $G$ contains a subgraph $H'$ isomorphic to $B_4$ such that $H\subset H'$.
\end{lemma}

\begin{proof}
First show that $H$ is an induced subgraph of $G$. Let $V(D_1)=\{u_1,u_2,u_3,u_4\}$, $V(D_2)=\{v_1,v_2,v_3,v_4\}$ and $u_1v_1,u_2v_2$ be the two edges linked $D_1$ to $D_2$. If there are two edges linking $\{v_3,v_4\}$ to $\{u_3,u_4\}$, then $G\cong I_6$. This contradicts Lemma \ref{counterex}. If there is exactly one edge linking $\{v_3,v_4\}$ to $\{u_3,u_4\}$, say $u_3v_3\in E(G)$, then $G[V(H)]\cong J_8$, a contradiction.

Since $G$ is claw-free, $H$ is an induced subgraph of $G$ and $G$ is not isomorphic to $I_7$, any two nonadjacent vertices of degree three in $H$ have no common neighbor.

Suppose that statement fails, that is, $u_3$ and $u_4$ ($v_3$ and $v_4$ respectively) have no common neighbor outside $H$. Let $w_1, w_2, w_3$ and $w_4$ are the four neighbors of $u_3, u_4, v_3$ and $v_4$ outside $H$, respectively. If $w_1w_2\in E(G)$ (or $w_3w_4\in E(G)$), then $w_1$ and $w_2$ (or $w_3$ and $w_4$) share two common neighbors outside $D_1$ (or $D_2$). Then $G$ contains a subgraph isomorphic to $L_3$, contradicting Lemma \ref{le1}. So suppose that neither $w_1w_2$ nor $w_3w_4$ is in $E(G)$.

If one of $w_1w_3$, $w_1w_4$, $w_2w_3$ and $w_3w_4$, say $w_1w_3$, is in $E(G)$, then $w_1$ and $w_3$ have two common neighbors $w'_1$ and $w'_3$. $\{w_1,w_3,w'_1,w'_3\}$ induces a $K_4$, say $D_3$. If $\{w_2,w_4\}\cap\{w'_1,w'_3\}\ne\emptyset$, then either $G\cong I_8$ or $G[V(H)\cup\{w_1,w_3,w'_1,w'_3\}]$ is isomorphic to $J_9$. If $\{w_2,w_4\}\cap\{w'_1,w'_3\}=\emptyset$, then $G[V(H)\cup\{w_1,w_2,w_3,w_4,w'_1,w'_3\}]$ is isomorphic to $J_{13}$. For both cases, we get a contradiction.

Now suppose that $W=\{w_1,w_2,w_3,w_4\}$ is an independent set.
Then $G[V(H)\cup W]$ is isomorphic to $J_{14}$. This contradicts Lemma \ref{forbsub}.
\end{proof}

\subsection{The proof of Theorem \ref{mainth}}                       

In this section, we present a proof of our main result, namely, Theorem \ref{mainth}.

 Let $T_1$ and $T_2$ be the graphs shown in Figure 7, $F_1$ and $F_2$ be the graphs shown in Figure 8. 
 For $i\in\{1,2\}$ and $j\in\{1,2,3\}$, let $T_i\circ A_j$ be the graph obtained from a $T_i$ and an $A_j$ by identifying a vertex of degree 2 in $T_i$ and a vertex of degree 2 in $A_j$. Since $G$ is a counterexample, it is easy to check that $G$ is not isomorphic to $T_1\circ T_1$.

\begin{center}
\setlength{\unitlength}{0.16mm}
\begin{picture}(413,222)
\put(251,139){\circle*{6}}
\put(251,78){\circle*{6}}
\put(379,139){\circle*{6}}
\put(379,78){\circle*{6}}
\put(316,140){\circle*{6}}
\put(251,107){\circle*{6}}
\put(318,77){\circle*{10}}\put(312,55){$u$}
\put(379,108){\circle*{6}}
\qbezier(251,139)(251,109)(251,78)
\qbezier(251,139)(315,139)(379,139)
\qbezier(251,78)(315,78)(379,78)
\qbezier(379,139)(379,109)(379,78)
\qbezier(316,140)(283,124)(251,107)
\qbezier(251,107)(284,92)(318,77)
\qbezier(316,140)(347,124)(379,108)
\qbezier(379,108)(348,93)(318,77)
\put(100,62){\circle*{6}}
\put(35,131){\circle*{6}}
\qbezier(100,62)(67,97)(35,131)
\put(165,130){\circle*{6}}
\qbezier(100,62)(132,96)(165,130)
\qbezier(35,131)(100,131)(165,130)
\put(70,95){\circle*{10}}\put(50,85){$u$}
\put(131,95){\circle*{6}}
\qbezier(70,95)(100,95)(131,95)
\put(102,131){\circle*{6}}
\qbezier(70,95)(86,113)(102,131)
\qbezier(131,95)(116,113)(102,131)
\put(62,208){\circle*{6}}
\put(62,162){\circle*{6}}
\put(142,208){\circle*{6}}
\put(142,162){\circle*{10}}\put(132,142){$v$}
\put(102,185){\circle*{6}}
\put(164,130){\circle*{6}}
\put(36,131){\circle*{6}}
\qbezier(62,208)(62,185)(62,162)
\qbezier(62,208)(102,208)(142,208)
\qbezier(62,162)(102,162)(142,162)
\qbezier(142,208)(142,185)(142,162)
\qbezier(62,208)(102,185)(142,162)
\qbezier(142,208)(102,185)(62,162)
\qbezier(164,130)(153,169)(142,208)
\qbezier(164,130)(153,146)(142,162)
\qbezier(62,208)(49,170)(36,131)
\qbezier(36,131)(49,147)(62,162)
\put(286,204){\circle*{6}}
\put(286,168){\circle*{6}}
\put(350,204){\circle*{6}}
\put(350,168){\circle*{10}}\put(340,148){$v$}
\put(318,186){\circle*{6}}
\put(379,139){\circle*{6}}
\put(251,139){\circle*{6}}
\qbezier(286,204)(286,186)(286,168)
\qbezier(286,204)(318,204)(350,204)
\qbezier(286,168)(318,168)(350,168)
\qbezier(350,204)(350,186)(350,168)
\qbezier(286,204)(318,186)(350,168)
\qbezier(350,204)(318,186)(286,168)
\qbezier(379,139)(364,172)(350,204)
\qbezier(379,139)(364,154)(350,168)
\qbezier(286,204)(268,172)(251,139)
\qbezier(251,139)(268,154)(286,168)
\put(85,29){$T_1$}
\put(303,26){$T_2$}
\put(100,-9){Figure 7. $T_1$ and $T_2$}
\end{picture}
\end{center}
\vskip .5em
\begin{center}
\setlength{\unitlength}{0.18mm}
\begin{picture}(700,185)
\put(124,170){\circle*{6}}
\put(124,94){\circle*{6}}
\qbezier(124,170)(124,132)(124,94)
\put(176,170){\circle*{6}}
\qbezier(124,170)(150,170)(176,170)
\put(176,94){\circle*{6}}
\qbezier(124,94)(150,94)(176,94)
\qbezier(176,170)(176,132)(176,94)
\qbezier(124,170)(150,132)(176,94)
\qbezier(176,170)(150,132)(124,94)
\put(218,130){\circle*{10}}\put(214,140){$u$}
\qbezier(218,130)(197,150)(176,170)
\put(633,180){\circle*{6}}
\put(585,132){\circle*{10}}\put(581,142){$u$}
\put(684,130){\circle*{6}}
\put(636,82){\circle*{6}}
\put(661,154){\circle*{6}}
\put(609,154){\circle*{6}}
\put(609,106){\circle*{6}}
\put(661,106){\circle*{6}}
\qbezier(633,180)(609,156)(585,132)
\qbezier(633,180)(658,155)(684,130)
\qbezier(585,132)(610,107)(636,82)
\qbezier(684,130)(660,106)(636,82)
\qbezier(661,154)(635,154)(609,154)
\qbezier(609,154)(609,130)(609,106)
\qbezier(661,154)(661,130)(661,106)
\qbezier(661,106)(635,106)(609,106)
\put(307,185){\circle*{6}}
\put(307,74){\circle*{6}}
\qbezier(307,185)(307,130)(307,74)
\put(218,129){\circle*{6}}
\qbezier(307,185)(262,157)(218,129)
\qbezier(307,74)(262,102)(218,129)
\qbezier(176,94)(197,112)(218,130)
\put(262,154){\circle*{6}}
\put(307,128){\circle*{6}}
\qbezier(262,154)(284,141)(307,128)
\put(262,105){\circle*{6}}
\qbezier(262,154)(262,130)(262,105)
\qbezier(262,105)(284,117)(307,128)
\put(69,170){\circle*{10}}\put(66,180){$v$}
\qbezier(124,170)(96,170)(69,170)
\put(69,94){\circle*{6}}
\qbezier(69,170)(69,132)(69,94)
\qbezier(69,94)(96,94)(124,94)
\put(13,170){\circle*{6}}
\put(13,94){\circle*{6}}
\put(69,170){\circle*{6}}
\put(69,94){\circle*{6}}
\qbezier(13,170)(13,132)(13,94)
\qbezier(13,170)(41,170)(69,170)
\qbezier(13,94)(41,94)(69,94)
\qbezier(13,170)(41,132)(69,94)
\qbezier(69,170)(41,132)(13,94)
\put(491,172){\circle*{6}}
\put(491,96){\circle*{6}}
\put(543,172){\circle*{6}}
\put(543,96){\circle*{6}}
\put(436,172){\circle*{6}}
\put(436,96){\circle*{6}}
\put(380,172){\circle*{6}}
\put(380,96){\circle*{6}}
\put(436,172){\circle*{10}}\put(433,182){$v$}
\put(436,96){\circle*{6}}
\qbezier(491,172)(491,134)(491,96)
\qbezier(491,172)(517,172)(543,172)
\qbezier(491,96)(517,96)(543,96)
\qbezier(543,172)(543,134)(543,96)
\qbezier(491,172)(517,134)(543,96)
\qbezier(543,172)(517,134)(491,96)
\put(585,132){\circle*{6}}
\qbezier(585,132)(564,152)(543,172)
\put(585,132){\circle*{6}}
\qbezier(543,96)(564,114)(585,132)
\qbezier(491,172)(463,172)(436,172)
\qbezier(436,172)(436,134)(436,96)
\qbezier(436,96)(463,96)(491,96)
\qbezier(380,172)(380,134)(380,96)
\qbezier(380,172)(408,172)(436,172)
\qbezier(380,96)(408,96)(436,96)
\qbezier(380,172)(408,134)(436,96)
\qbezier(436,172)(408,134)(380,96)
\put(160,50){$F_1$}\put(500,50){$F_2$}
\put(260,5){Figure 8. $F_1$ and $F_2$.}
\end{picture}
\end{center}

We give the following order to choose a packing $\mathcal{P}_0$ for $G$.

\noindent
{\bf Initialize.} $\mathcal{P}_0=\emptyset$.

\noindent
{\bf Step 1.} If $G$ contains $F_i$ for some $i=1,2$ and none saturated  vertex of $F_i$ has message, then we add the focal vertex $u$ and a saturated vertex $v$ in $F_i$ with $d_{F_i}(u,v)=3$ to $\mathcal{P}_0$ (see, the two bigger vertices in Figure 8). Process Step 1 till $G$ contains no such an $F_i$. Then go to Step 2.

\noindent
{\bf Step 2.} If $G$ contains $T_i\circ A_j$ for some $i\in\{1,2\},j\in\{1,2,3\}$ and none saturated vertex of $T_i\circ A_j$ has message, then we add the focal vertex (linked $T_i$ and $A_j$) $u$ and a saturated vertex $v$ with $d_{T_i}(u,v)=3$ to $\mathcal{P}_0$. Process Step 2 till $G$ contains no such a $T_i\circ A_j$. Then go to Step 3.

\noindent
{\bf Step 3.} If $G$ contains $T_i$ for some $i=1,2$ and none saturated vertex of $T_i$ has message, then we add two indictor vertices $u$ and $v$ with $d_{T_i}(u,v)=3$ to $\mathcal{P}_0$ (see, the two bigger vertices in Figure 7). Process the step till $G$ contains no such a $T_i$. Then go to Step 4.

\noindent
{\bf Step 4.} If $G$ contains $A_i\circ A_j$ for $i, j\in\{1,2,3\}$ and none saturated vertex of $A_i\circ A_j$ has message, then we add the focal vertex to $\mathcal{P}_0$. Process the step till $G$ contains no such an $A_i\circ A_j$. Then go to Step 5.

\noindent
{\bf Step 5.} If $G$ contains $A_i$ for some $i\in \{1, 2, 3\}$ and none saturated vertex of $A_i$ has message, then we add one saturated vertex to $\mathcal{P}_0$. Process the step till $G$ contains no such an $A_i$.

\noindent
{\bf Output.} $\mathcal{P}_0$.

\noindent
{\bf Remark.} Notice that every vertex of $P_0$  is either a saturated vertex or a focal vertex in some subgraph of $G$, and before the vertex is chosen to $P_0$, it does not have message.

\begin{lemma}\label{55}
Let $H$ be a subgraph isomorphic to some $A_i$ for $i=2,3$. Then $|V(H)\cap \mathcal{P}_0|=1$.
\end{lemma}

\begin{proof}
 Obviously $|V(H)\cap \mathcal{P}_0|\geq 1$. Suppose $x\in V(H)\cap \mathcal{P}_0$. 
 By the remark above, at least one saturated vertex of $H$ has message. By Lemma \ref{le4}, $H$ is an induced subgraph of $G$. By Lemma \ref{le3}, if $H$ has a common vertex with other subgraph isomorphic to $A_i$ for $i\in \{1,2,3\}$, then the common vertex is the vertex of degree two of $H$. Therefore none vertices of $H$ is the saturated vertex in other subgraph isomorphic to $A_i$ for some $i$. According to the above choice order, we do not choose any vertex in $H$ but $x$. Hence $|V(H)\cap \mathcal{P}_0|=1$.
\end{proof}

\begin{lemma}\label{5}
 Let $\mathcal{P}_0$ be the vertex subset of $G$ obtained by the choice order above. Then $\mathcal{P}_0$ is a packing of $G$.
\end{lemma}

\begin{proof}
Suppose the statement is false. That is, there are two vertices $x$ and $y$ of $\mathcal{P}_0$ such that $d(x,y)\le 2$. In fact, we only need to consider the case for $d(x,y)=2$. We partition our discussion into two cases.

{\bf Case 1.} $x$ (or $y$) is a focal vertex of some subgraph $H\cong A_i\circ A_j$ of $G$, where $i,j\in\{1,2,3\}$.

Since $d(x,y)=2$, without loss of generality say that both of $x$ and $y$ are in $A_i$ of $H$. Now $|V(H)\cap \mathcal{P}_0|\ge 2$, by Lemma \ref{55}, we have $A_i=A_1$. Denote the subgraph isomorphic to $A_1$ in $H$ by $H_1$. Then the two vertices of degree three in $H_1$ are the saturated vertices in other copies of $A_k$ for some $k\in\{1,2,3\}$ or $B_4$. By Lemmas \ref{forbsub} and \ref{le2}, there is a subgraph $H_2$ isomorphic to $B_3$ or $B_4$ such that $H_1\subset H_2$.  If $y$ is chosen prior to $x$, then all the vertices of $A_1$ have message. It means that $x\notin P_0$, a contradiction. Now we suppose that $x$ is chosen prior to $y$.

If $H_2\cong B_3$, then by Lemma \ref{le61}, there is a subgraph $H_3$ isomorphic to $B_1$ or $B_2$ such that $x,y\in V(H_3)$. Since $x\in V(A_1\circ A_j)$ and based on the configuration of $A_1$, $A_j$ is isomorphic to $A_2$ or $A_3$. By Lemma \ref{le4}, $A_j$ is an induced subgraph of $G$. If $H_3$ is isomorphic to $B_2$, then one vertex of of degree two in $H_2$ is a saturated vertex of  $A_j$. This contradicts Lemma \ref{le3}. If $H_3$ is isomorphic to $B_1$,
then $H_2$ and $A_j$ have two common vertices of degree two. 
 By the configuration of $A_2$ or $A_3$, there is a subgraph $H_4$ isomorphic to $T_1$ or $T_2$ such that $x,y\in V(H_4)$.

Note that all the saturated vertices of $H_4$ have no message before $x,y$ are chosen. If there is a subgraph $H_5$ isomorphic to $F_1$ or $F_2$ such that $V(H_4)\cap V(H_5)$ isomorphic to $A_2$ or $A_3$, and none saturated vertices of $H_5$ has message before $x,y$ are chosen, then the focal vertex in $H_5$ which is also the vertex of degree two in $H_4$ should be chosen in Step 1. If there is no such an $H_5$, then we choose vertices of $H_3$ by Steps 2 or 3. For both cases, we do not choose $x$ and $y$ in $H_2$.

Now we consider that $H_2\cong B_4$. Note that $H$ and $H_2$ share an $A_1$. If $H\cong A_1\circ A_1$, then $H$ is also an induced subgraph of $G$, contradicting Lemma \ref{forbsub}. If $H\cong A_1\circ A_2$ (or $H\cong A_1\circ A_3$), then both of $x$ and $y$ are in a subgraph $H_3$ isomorphic $F_1$ or $F_2$. Since $x$ is chosen prior to $y$, all the saturated vertices of $H_3$ have no message before $x$ and $y$ are chose. Then we do not choose $y$ in $H_3$, a contradiction.

{\bf Case 2.} Neither $x$ nor $y$ is a focal vertex of any subgraph isomorphic to $A_i\circ A_j$ of $G$, where $i,j\in\{1,2,3\}$.

If $x$ is in some subgraph $H\cong F_i$ of $G$, where $i\in\{1,2\}$. If $y\in V(H)$, then $x$ is prior chosen to $y$. Note that $x$ and the focal vertex in $H$ are chosen to $P_0$ in Step 1 at the same time. Then $y\notin P_0$, a contradiction. Suppose that $y\notin V(H)$. Since $y$ is not the focal vertex, by Lemma \ref{le1}, $y$ is a vertex of degree three in a copy of $A_1$, a contradiction.

 If $x$ is in a copy of $A_i$ for some $i\in\{2,3\}$, then by Lemma\ref{55}, $y$ is in another copy of $A_j$ for some $j\in\{1,2,3\}$. Since $d(x,y)=2$, both of $x$ and $y$ are in a subgraph $H$ isomorphic to $A_i\circ A_j$.
 Before $x,y$ being chosen, all the saturated vertices of $H$ have no message. If both of $x$ and $y$ are in some subgraph isomorphic to $T_1$ or $T_2$, then by the similar discussion as Case 1, we can get a contradiction. Otherwise, we choose the focal vertex other than $x$ or $y$ in $H$. A contradiction.

Now we suppose that $x$ is in the copy of $A_1$. Base on the above discussion, we only need to consider that $y$ is in another copy of $A_1$. If $x, y$ are in a subgraph $H$ isomorphic to $A_1\circ A_1$, then by Lemma \ref{forbsub}, $G[V(H)]\cong I_7$, a contradiction. By Lemmas \ref{le2}, we only need to consider that $x$ and $y$ are in a subgraph $H$ isomorphic to $B_3$ or $B_4$.

If $H\cong B_3$, then $y$ is the focal vertex, a contradiction. If $H\cong B_4$, then both of $x$ and $y$ are in a subgraph $H_1$ isomorphic to $F_1$ or $F_2$, where all the saturated vertices of $H_1$ have no message before $y$ being chosen. Note that $y$ is chosen prior to $x$. By Step 1, we do not choose $x$ in $H_1$, a contradiction.

This completes the proof.
\end{proof}

Recall that for each integer $i\ge 0$ and a vertex set $S$ of $G$, $P^0(S)=S$ and $P^{i+1}(S)=\cup \{N_G[v]: v\in P^{i}(S)$ such that $|N_G[v] \setminus P^{i}(S)|\leq 1\}$. Note that if $H$ is isomorphic to $A_i$ for some $i\in\{1,2,3\}$, then from Step 5 of the choice order $P^{1}(\mathcal{P}_0)$ contains at least one saturated vertex of $H$.

\begin{lemma}\label{lex}
Let $G$ contain a subgraph $H$ isomorphic to $L_2$ and $\mathcal{P}_0$ be a packing  denoted as above. Then $P^{1}(\mathcal{P}_0)$ contains at least one saturated vertex of $H$.
\end{lemma}

\begin{proof}
By Lemma \ref{le62}, there is a subgraph $H_1$ isomorphic to $B_4$ such that $H\subset H_1$. If $H_1$ is contained in some subgraph isomorphic to $F_1$ or $F_2$, then by Step 1, all the saturated vertices of $H$ have message. Otherwise, by Step 5, at least one saturated vertices of $H$ could obtain message.
\end{proof}

We extend the packing $\mathcal{P}_0$ of $G$ to a maximal packing and denote the resulting packing by $S_0$.

\begin{lemma}\label{51}
$G$ has a sequence $S_0\subset S_1\subset \cdots\subset S_l$ such that the following holds:

(i) For all $0\le i\le l-1, |S_{i+1}|=|S_i|+1$ and $|P^\infty(S_{i+1})|\geq|P^\infty(S_i)|+5$.

(ii)$P^\infty(S_l)=V$.
\end{lemma}

\begin{proof}
Note that if $P^{\infty}(S_0)=V(G)$, then we are done. Now suppose that there is an $S_{i}$ for some $i\geq 0$ such that $M=P^\infty (S_i)$ does not contain all vertices of $G$. Denote $\overline{M}=V \setminus M$. Let $\mathcal{U}=\{u: u\in M, N_G(u)\setminus M\neq \emptyset\}$. It is obvious that $d_{M}(u)\geq1$ and $2\le d_{\overline{M}}(u)\le 3$ for each vertex $u\in \mathcal{U}$. We have the following statements.

\noindent
{\bf Claim 1.} For each $u\in \mathcal{U}$, $N_G(u)\setminus M$ induces a clique in $G$.

\begin{proof}
Suppose $x_1$ and $x_2$ are two neighbors of $u$ in $N_G(u)\setminus M$ and $u$ receives message from $v$ in $M$. Then $x_1v,x_2v\notin E(G)$. If $x_1x_2\notin E(G)$, then $\{u,x_1,x_2, v\}$ induces a claw. A contradiction.
\end{proof}

\noindent
{\bf Claim 2.} For each vertex $x\in\overline{M}$, $d_{\overline{M}}(x)\geq 2$.

\begin{proof} If all neighbors of $x$ are in $\overline{M}$, then we are done. Now
suppose that $x$ has a neighbor $x'\in M$. Note that $x'\in \mathcal{U}$. Since $2\le d_{\overline{M}}(x')\le 3$ and $N_G(x')\setminus M$ induces a clique, $d_{\overline{M}}(x)\ge 1$. If $d_{\overline{M}}(x)=1$, then let $x_1$ be the neighbor of $x$ in $\overline{M}$ and $N_G(x)\cap M=\{y_1, y_2, x'\}$. Since none of vertices in $ \{y_1, y_2, x'\}$ sends message to $x$ and $N_G(y_i)\setminus M$ induces a clique for each $i\in\{1,2\}$, $N_M(x_1)=\{y_1, y_2, x'\}$. Since $G$ is claw-free, $G[\{y_1, y_2, x'\}]$ contains at least one edge. Therefore $G[\{x,x_1,y_1, y_2, x'\}]$ contains an $A_1$ with two saturated vertices $x$ and $x_1$. Note that the two vertices have no message, contradicting to Step 5 of choice order.
\end{proof}

\noindent
 {\bf Claim 3.} Let $H\cong K_4$ be a subgraph of $G$ and $x$ be a vertex of $H$. If $x\in\overline M$ and there is a vertex $v$ such that $N_H(v)=\{x\}$, then $v\in\overline M$.

\begin{proof}
Since $G$ is claw-free and $N_H(v)=\{x\}$, $N_G[v]\setminus\{x\}$ induces a subgraph $H'$ isomorphic to $K_4$. If $v\in M$, then $v$ will receive message from one of its neighbors in $H'$. Since $H'\cong K_4$, all vertices of $H'$ have message. Now $v$ can send message to $x$, contradicting that $x\in\overline M$.
\end{proof}

Note that $2\le d_{\overline{M}}(u)\le 3$ for each vertex $u\in \mathcal{U}$. If there is a vertex $u\in \mathcal{U}$ such that $d_{\overline{M}}(u)=3$, then $u$ together with its three neighbors in $\overline{M}$, say $u_1,u_2,u_3$, induces a $K_4$. If two vertices of the three vertices, say $u_1$ and $u_2$, have a common vertex $u_4$ other than $u$ or $u_3$, then $\{u, u_1, u_2, u_3, u_4\}$ induces an $A_1$ with its saturated vertices $u_1$ and $u_2$. Note that neither $u_1$ nor  $u_2$ has message. This contradicts Lemma \ref{lex}. Now suppose that each vertex in $\{u_1,u_2,u_3\}$ is adjacent exactly one vertex in $\overline{M}\setminus \{u_1,u_2,u_3\}$. Then let $S_{i+1}=S_{i}\cup \{u_1\}$. We have $|P^\infty(S_{i+1})|\geq |P^\infty(S_i)|+5$.

 Then we suppose that $d_{\overline M}(u)=2$ for each $u\in \mathcal{U}$. Let $u\in \mathcal{U}$, $N_G(u)=\{v, v', u_1, u_2\}$, $v$ sends message to $u$ and $u_1, u_2\in \overline M$. We consider the two cases as follows.

 {\bf Case 1.}  $vv'\notin E(G)$.

 By the claw-freeness of $G$, $\{u_1, u_2, v', u\}$ induces a $K_4$. If $u_1$ and $u_2$ share a neighbor other than $v',u$, say $w$, then $\{u_1, u_2, u, v', w\}$ induces an $A_1$ such that its two saturated vertices $u_1$ and $u_2$ have no message, a contradiction. By Claim 3, $u_1$ and $u_2$ are adjacent to two distinct vertices $w_1$ and $w_2$ in $\overline M\setminus \{v',u,u_1,u_2\}$ respectively. If $w_1w_2\in E(G)$, then by the claw-freeness of $G$, both of $w_1$ and $w_2$ are in another copy of $K_4$. Then $\{u_1, u_2, w_1,w_2\}$
induces a saturated cycle in a copy of $L_2$, where none of saturated vertices of the $L_2$ has message, a contradiction. If $w_1w_2\notin E$, then by Claim 2, $w_1$ has a neighbor $w_3\in \overline M$ other than $u_1$. We define $S_{i+1}=S_{i}\cup \{w_1\}$. It is easy to check that $|P^\infty(S_{i+1})|\geq |P^\infty(S_i)|+5$.

{\bf Case 2.} $vv'\in E(G)$.

Suppose that both of $u_1$ and $u_2$ are adjacent to $v'$. Then whether $u_1$ and $u_2$ share a neighbor in $\overline M$ or not, by the similar discussion as Case 1, we can obtain a contradiction or a desired $S_{i+1}$.

Now we suppose that neither $u_1$ nor $u_2$ is adjacent to $v'$. If $u_1$ and $u_2$ share two neighbors other than $u$, then $G[N[u_1]]\cong A_1$ and none of two saturated vertices of $G[N[u_1]]$ has message, a contradiction. If $u_1$ and $u_2$ share exactly one neighbor other than $u$, say $w$, then let $w_1=N(u_1)\setminus\{w,u,u_2\}$ and $w_2=N(u_2)\setminus\{w,u,u_1\}$. Since $G$ is claw-free, $w_1w, w_2w\in E(G)$. Note that $w$ has no message (otherwise, $w$ receives message from $w_1$ or $w_2$, but it is impossible). Now $\{u,u_1,u_2,w,w_1,w_2\}$ induces an $A_2$ such that its saturated vertices $u_1,u_2,w$ have no message. A contradiction. If $u_1$ and $u_2$ do not have other common neighbors than $u$, then let $\{w_1,w_2\}=N(u_1)\setminus\{u,u_2\}$ and $\{w_3,w_4\}=N(u_2)\setminus\{u,u_1\}$. Since $G$ is claw-free, $w_1w_2, w_3w_4\in E(G)$. By Claim 2, we suppose without loss of generality that $w_1,w_3\in\overline M$.
First suppose that $w_2\in \overline M$.
If the two vertices of $N(w_1)\setminus\{w_2,u_1\}$ are also in $M$, then by Claim 1, $N[w_1]$ induces an $A_1$ such that none of its saturated vertices ($w_1$ and $w_2$) has message, a contradiction. Then at least one vertex of $N(w_1)\setminus\{w_2,u_1\}$ is in $\overline M$.  Let  $S_{i+1}=S_{i}\cup \{w_1\}$. We have  $|P^\infty(S_{i+1})| \geq |P^\infty(S_i)|+5$.
 The case for $w_4\in \overline M$ is similar completely.
Now suppose that both of $w_2$ and $w_4$ are in $M$. By Claim 2, there is a vertex $t$ in $\overline M$ adjacent to $w_1$. If $t=w_3$, then by the claw-freeness of $G$, $w_1$ and $w_3$ share a neighbor, say $w_5$. Now $\{u,u_1,u_2,w_1,w_2,w_3,w_4,w_5\}$ induces an $A_3$. Since all vertices in the saturated cycle $u_1u_2w_1w_3u_1$ have no message, we get a contradiction. If $t\ne w_3$ then let  $S_{i+1}=S_{i}\cup \{w_1\}$. Thus, $|P^\infty(S_{i+1})|\geq |P^\infty(S_i)|+5$.

Now we consider the case for exactly one of $u_1v'$ and $u_2v'$ is in $E(G)$. Without loss of generality say $u_1v'\in E(G)$ and $u_2v'\notin E(G)$. By the front discussion, we can suppose that $d_{\overline M}(v')=2$ and $u_3$ is the other neighbor of $v'$ in $\overline{M}$. By the claw-freeness of $G$, $u_1u_3\in E(G)$. Then $N[u]\cap\{u_3\}$ induces an $A_2$. Note that $u_2u_3\notin E(G)$ since each subgraph isomorphic to $A_2$ is also an induced subgraph of $G$. By Claim 2, we have $2 \le d_{\overline M}(u_2)\le 3$. If $d_{\overline M}(u_2)=3$, then let $S_{i+1}=S_{i}\cup \{u_2\}$. If $d_{\overline M}(u_3)=3$ then let $S_{i+1}=S_{i}\cup \{u_3\}$. For both cases,  $|P^\infty(S_{i+1})| \geq |P^\infty(S_i)|+5$. Now suppose that $d_{\overline M}(u_2)=d_{\overline M}(u_3)=2$. Let $w_2, w_3\in \overline M $ be another neighbors of $u_2$ and $u_3$ respectively. If $w_2=w_3$, then similar to the discussion in the paragraph above, we obtain a subgraph isomorphic to $A_3$ with all its saturated vertices have no message, a contradiction. If $w_2\neq w_3$, then let $S_{i+1}=S_{i}\cup \{u_1\}$. Thus, $|P^\infty(S_{i+1})| \geq |P^\infty(S_i)|+5$.

Since $|V|$ is finite, there exists an integer $l$ such that $P^\infty (S_l) = V$.
\end{proof}

We are now in a position to prove our main result, namely Theorem \ref{mainth}.

\begin{proof} Let $G$ be a counterexample such that $|V(G)|$ is minimal. Let $S_0,S_1,\cdots,S_l$ be a desired sequence of Lemma \ref{51}. Then $\gamma_{P}(G)\leq|S_l|$. Since $S_0$ is a packing in $G$, we have $|P^\infty(S_0)|=|N[S_0]|=5|S_0|$. If $l=0$ then $n=5|S_0|$ and $\gamma_{p}(G)\leq|S_0|\le\frac{n+1}{5}$. We are done. For $1\le i\le l$, by Lemma \ref{51}(i), we have  $|S_i|=|S_{i-1}|+1$. It means that $P^\infty (S_i )\geq P^\infty (S_{i-1})+5$. Then
$n = |P^\infty (S_l )|\geq |P^0(S_0)|+5l = 5(|S_0| + l) \geq 5\gamma_p(G)$,
Thus $\gamma_p(G) \leq n/5$. It contradicts that $G$ is a counterexample.
This completes the proof.
 \end{proof}

\section{Linear-time algorithm for power domination in weighted trees}

The weighted domination problem has been well-studied in the last several decades. Farber \cite{Farber81}, Natarajan and White \cite{Nat78} independently studied the classical domination in weighted trees. There is an extensive number of papers concerning the algorithmic complexity of the weighted domination problem in graphs, such as distance-hereditary graphs \cite{Yeh98}, chordal graphs \cite{Chang04}, interval graphs \cite{Bertossi88, Ramalingam88}. We refer to \cite{C1998} for more results and details.

An natural extension of power domination is \emph{weighted
power-domination}. Let $G=(V,E,w)$ be a weighted graph, where $w$ is a function from $V$ to positive real
numbers. Let $w(S)=\sum_{v\in S}w(v)$ be the weight of $S$ for any
subset $S$ of $V$. The \emph{weighted power domination number},
denoted by $\gamma_{p}^w(G)$, is defined as
$\gamma_{p}^w(G)=\min\{w(S)~|~S$ is a power dominating set of
$G\}$.
The \emph{weighted power domination problem} is to determine the weighted power domination number of
any weighted graph.

 Let $T=(V,E)$ be a tree with $n$ vertices. It is well known that the vertices of $T$ have an ordering  $v_1,v_2,\cdots, v_n$
 such that for each $1\le i\le n-1$, $v_i$ is adjacent to exactly one $v_j$ with $j>i$. The ordering is call a \emph{tree ordering} of the tree,
 where the only neighbor $v_j$ with $j>i$ is called the \emph{father} of $v_i$ and $v_i$ is a \emph{child} of $v_j$. For each $1\le i\le n-1$, the father of $v_i$ is denoted by $F(v_i)=v_j$. For technical reasons, we assume that $F(v_n)=v_n$.

In this section, a linear time dynamic programming style algorithm
is given to compute the exact value of weighted power domination
number in any tree. This algorithm is
constructed using the methodology of Wimer \cite{WI1987}.

We make use of the fact that the class of rooted tree can be
constructed recursively from copies of the single vertex $K_1$,
using only one rule of composition, which combines two trees
$(T_1,r_1)$ and $(T_2,r_2)$ by adding an edge between $r_1$ and
$r_2$ and calling $r_1$ the root of the resulting larger tree $T$.
We denote this as $(T,r_1)=(T_1,r_1)\circ (T_2,r_2)$.

In particular, if $D$ is a power dominating set of $T$, then $D$
splits two subsets $D_1$ and $D_2$ according to this decomposition.
However, $D_1$~($D_2$,~respectively) may not be a power dominating
set of $T_1$~($T_2$,~respectively). We express this as follows:
$(T,D)=(T_1,D_1)\circ (T_2,D_2)$. Let $T$ be a tree rooted at $r$.
 $\bar T$ is a new tree rooted at $r'$, where $V(\bar T)=V(T)\cup \{r'\}$ and $E(\bar T)=E(T)\cup \{rr'\}$.

In order to construct an algorithm to compute weighted
power domination number, we must characterize the possible
tree-subset pairs $(T,D)$. For this problem there are five
classes:\\
$[a]=\{(T,D)~|~D$ is a PDS of $T$ and $r\in D\}$;\\
$[b]=\{(T,D)~|~D$ is a PDS of $\bar T$ and $r\not\in D\}$;\\
$[c]=\{(T,D)~|~D$ is a PDS of $T$, but not of $\bar T$ and $r\not\in D\}$;\\
$[d]=\{(T,D)~|~D$ is a PDS of $T-r$, but it is not a PDS of $T$ and $r\not\in D\}$;\\
$[e]=\{(T,D)~|~D$ is not a PDS of both $T$ and $T-r$, but all vertices of $T$ can be observed by $D$ if a message is given to $r$ in advance$\}$.

\vskip 1cm
\begin{center}
\setlength{\tabcolsep}{4mm}

\begin{tabular}{c|ccccc}
$\circ$ &$[a]$&$[b]$&$[c]$&$[d]$&$[e]$\\\hline
$[a]$&  $[a]$   &  $[a]$   &  $[a]$     &  $[a]$   &  $[a]$  \\[10pt]
$[b]$&  $[b]$   &  $[b]$   &  $[b]$  &  $[c]$ &  $[c]$   \\[10pt]
$[c]$&  $[c]$   &  $[c]$   &  $[c]$  &  $\times$   &  $\times$   \\[10pt]
$[d]$&  $[b]$   &  $[b]$   &  $[d]$   &  $[e]$ & $[e]$  \\[10pt]
$[e]$&  $[c]$   &  $[c]$   &  $[e]$   & $\times$ & $\times$ \\[10pt]
\end{tabular}
\end{center}

Next, we must consider the expression $(T_1,D_1)\circ (T_2,D_2)$
with $(T_1,D_1)$ of class $[i]$ and $(T_2,D_2)$ of class $[j]$,
where $i,j\in \{a,b,c,d,e\}$. The above table shows the results of
$(T_1,D_1)\circ (T_2,D_2)$ in every possible cases. From this table, we obtain\\
$[a]=[a]\circ[a]\cup [a]\circ [b]\cup [a]\circ [c]\cup [a]\circ [d]\cup [a]\circ [e]$;\\
$[b]=[b]\circ [a]\cup [b]\circ [b]\cup [b]\circ [c]\cup [d]\circ [a]\cup [d]\circ [b]$;\\
$[c]=[b]\circ [d]\cup [b]\circ [e]\cup [c]\circ [a]\cup [c]\circ [b]\cup [c]\circ [b]\cup [e]\circ [a]\cup [e]\circ [b]$;\\
$[d]=[d]\circ [c]$;\\
$[e]=[d]\circ [d]\cup [d]\circ [e]\cup [e]\circ [d].$

The above formulation means that, for example, $(T,D)$ of class
$[a]$ can be obtained from $(T_1,D_1)$ of class $[a]$ and
$(T_2,D_2)$ of class $[a]$, or $(T_1,D_1)$ of class $[a]$ and
$(T_2,D_2)$ of class $[b]$, or $(T_1,D_1)$ of class $[a]$ and
$(T_2,D_2)$ of class $[c]$, or $(T_1,D_1)$ of class $[a]$ and
$(T_2,D_2)$ of class $[d]$, or $(T_1,D_1)$ of class $[a]$ and
$(T_2,D_2)$ of class $[e]$. It is easy to check that the above
formulation is correct by inspection. The final step is to define
the initial vector. In this case, for trees, the only basis graph is
the tree with single vertex $v$. It is easy to obtain that the
initial vector is $(w(v),\infty,\infty,0,\infty)$, where $'\infty'$ means
undefined. Now, we are ready to present the algorithm.

\vspace{4mm}

\noindent{\bf Algorithm WPDT.} Compute the weighted power domination number for a weight tree.\\
{\bf Input:} A tree $T=(V,E,w)$ with a tree
ordering $v_1,v_2,\cdots,v_n$.\\
{\bf Output:} The weighted power domination number $\gamma_{p}^w(T)$.\\
$1.$ {\bf for} $i:=1$ {\bf to} $n$ {\bf do}\\
$2.$ initialize vector $[i,1..5]$ to $[w(v_i),\infty,\infty,0,\infty]$;\\
$3.$ {\bf endfor}\\
$4.$ {\bf for} $j:=1$ {\bf to} $n-1$ {\bf do}\\
$5.$ $v_k=F(v_j)$;\\
$6.$ vector$[k,1]:=$vector$[k,1]+\min\limits_{1\le m\le 5}$vector$[j,m]$; \\
$7.$ vector$[k,2]:=\min\{$vector$[k,2]+\min\limits_{1\le m\le 3}$vector$[j,m]$,
vector$[k,4]+\min\limits_{1\le m\le 2}$vector$[j,m]\}$;\\
$8.$ vector$[k,3]:=\min\{$vector$[k,2]+\min\limits_{4\le m\le 5}$vector$[j,m]$,vector$[k,3]+\min\limits_{1\le m\le 3}$vector$[j,m]$,\\
\hspace*{40mm} vector$[k,5]+\min\limits_{1\le m\le 2}$vector$[j,m]$\};\\
$9.$ vector$[k,4]:=$vector$[k,4]+$vector$[j,3]$;\\
$10.$ vector$[k,5]:=\min\{$vector$[k,4]+\min\limits_{4\le m\le 5}$vector$[j,m]$, vector$[k,5]+$vector$[j,3]\}$.\\
$11.$ {\bf endfor}\\
$12.$ $\gamma_{p}^w(T)=\min\{$vector$[n,1]$, vector$[n,2]$, vector$[n,3]\}$.

\vspace{6mm}

From the above argument, we can obtain the following theorem.

\begin{theorem}
Algorithm $WPDT$ can output the
weighted power domination number of any weighted tree $T=(V,E,w)$
 in linear time $O(m+n)$, where $n=|V|$ and $m=|E|$.
\end{theorem}

\end{document}